\documentclass[a4paper,11pt,fleqn,reqno,oneside]{amsart}
\usepackage{lmodern}
\usepackage{enumerate}

\usepackage{mathtools}

\usepackage{amsmath}
\usepackage{amssymb}
\usepackage{amsthm}
\usepackage{amsfonts}

\usepackage{mathrsfs} 
\usepackage[english]{babel}
\usepackage{esvect}
\usepackage{extarrows}
\usepackage[noadjust]{cite}

\let\Gamma\varGamma

\usepackage{xcolor}

\usepackage[
	paper = a4paper,
	top = 32mm, bottom = 32mm,
	left = 34mm, right = 34mm,
	footskip = 10mm, footnotesep=5mm, 
]
{geometry}

\usepackage{doi}

\usepackage[sorted-cites]{amsrefs}

\usepackage{etoolbox}
\patchcmd{\section}{\scshape}{\bfseries}{}{}
\makeatletter
\renewcommand{\@secnumfont}{\bfseries}
\patchcmd{\@startsection}
  {\@afterindenttrue}
  {\@afterindentfalse}
  {}{}
\makeatother

\newcommand{\Rb}{\mathbb{R}}

\newcommand{\ee}{\mathrm{e}}

\newcommand{\<}{\langle}
\renewcommand{\>}{\rangle}

\newcommand{\der}{\mathrm{d}}

\renewcommand{\div}{\mathrm{div}}

\newcommand{\gM}{\s}

\DeclareMathOperator{\Ric}{Ric}

\newcommand{\gind}{g}

\newcommand{\piR}{\pi_{\Rb}}

\newcommand{\A}{A}
\newcommand{\Hv}{\vv{H}} 			

\DeclareSymbolFont{UPM}{U}{eur}{m}{n}
\DeclareMathSymbol{\partial}{0}{UPM}{"40}

\usepackage{enumitem}

\newtheorem{theorem}{Theorem}
\newtheorem{lemma}[theorem]{Lemma}
\newtheorem{corollary}[theorem]{Corollary}
\newtheorem{proposition}[theorem]{Proposition}

\theoremstyle{definition}
\newtheorem{remark}[theorem]{Remark}
\newtheorem*{remark*}{Remark}
\newtheorem{definition}[theorem]{Definition}
\newtheorem*{definition*}{Definition}

\numberwithin{equation}{section}
\numberwithin{theorem}{section}

\DeclareMathOperator{\Max}{max}

\newcommand{\wh}{\widehat}
\newcommand{\wt}{\widetilde}

\renewcommand{\d}{\delta}
\renewcommand{\l}{\lambda}

\newcommand{\ve}{\varepsilon}

\newcommand{\s}{\sigma}

\global\long\def\norm#1{\left\Vert #1\right\Vert }
\global\long\def\abs#1{\left|#1\right|}

\renewcommand{\r}{\mathbb{R}} 
\newcommand{\N}{\mathbb{N}}

\DeclareMathOperator{\graph}{graph}

\renewcommand{\le}{\leq}
\newcommand{\p}{\partial}

\newcommand{\n}{\nabla}

\newcommand{\ChrM}{\prescript{\s}{}\!\Gamma} 

\def\R{\mathbb{R}}

\def\eps{\varepsilon}

\renewcommand{\d}{\mathrm{d}}

\allowdisplaybreaks

\title[Mean curvature flow in asymptotically flat product spacetimes]{Mean curvature flow in\\asymptotically flat product spacetimes}

\author{Klaus Kr\"oncke}
\author{Oliver Lindblad Petersen}
\author{Felix Lubbe}
\author{Tobias Marxen}
\author{Wolfgang Maurer}
\author{Wolfgang Meiser}
\author{Oliver C.\ Schn\"{u}rer}
\author{\'Aron Szab\'o}
\author{Boris Vertman}

\address{Klaus Kr\"oncke, \'Aron Szab\'o: Universit\"at Hamburg, Fachbereich Mathematik, Bundesstr.\ 55, D-20146 Hamburg}
\email{klaus.kroencke@uni-hamburg.de, aron.szabo@uni-hamburg.de}

\address{Oliver Lindblad Petersen: Stanford University, Department of Mathematics, 450 Jane Stanford Way, Stanford, California 94305}
\email{oliverlp@stanford.edu}

\address{Felix Lubbe: University of Copenhagen, Department of Mathematical Sciences, Universitetsparken 5, DK-2100 Copenhagen {\O}}
\email{felix-lubbe@math.ku.dk}

\address{Tobias Marxen, Boris Vertman: Universit\"at Oldenburg, Institut f\"ur Mathematik, Carl-von-Ossietzky-Str.\ 9--11, D-26129 Oldenburg}
\email{tobias.marxen@uol.de, boris.vertman@uol.de}

\address{Wolfgang Maurer, Oliver C.\ Schn\"urer: Universit\"at Konstanz, Fachbereich Mathematik und Statistik, Universit\"atsstr.\ 10, D-78464 Konstanz}
\email{wolfgang.maurer@uni-konstanz.de, oliver.schnuerer@uni-konstanz.de}

\address{Wolfgang Meiser: Universit\"at Magdeburg, Institut f\"ur Analysis und Numerik, Universit\"atsplatz 2, D-39106 Magdeburg}
\email{wolfgang.meiser@uvgu.de}

\makeatletter
\@namedef{subjclassname@2020}{%
	\textup{2020} Mathematics Subject Classification}
\makeatother

\date{\today}

\begin{document}

\begin{abstract}
	We consider the long-time behaviour of the mean curvature flow of spacelike hypersurfaces in the Lorentzian product manifold $M\times\Rb$,
	where $M$ is asymptotically flat.
	If the initial hypersurface $F_0\subset M\times \R$ is uniformly spacelike and asymptotic to $M\times\left\{s\right\}$ for some $s\in\R$ at infinity, we show that a mean curvature flow starting at $F_0$ exists for all times and converges uniformly to $M\times\left\{s\right\}$ as $t\to \infty$.
\end{abstract}

\maketitle
\tableofcontents

\section{Introduction and main results}

The mean curvature flow is a $1$-parameter family $F_t$ of embedded submanifolds of a semi-Riemannian manifold that evolves in the direction of its mean curvature vector, that is, $F_t$ satisfies
\[
\partial_t F(x,t) = \Hv(x,t) \,, \qquad F(x,0) = F_0(x)
\]
for all $x\in M$ and $t\in[0,T)$ for some $T>0$.
Especially the case of hypersurfaces of Riemannian manifolds has been studied extensively in the literature, see e.\,g.\ \cite{colding2015mean} for a recent survey.

Another interesting case that has been studied less often is the case of hypersurfaces in globally hyperbolic Lorentzian manifolds. 
In this case,
the mean curvature flow is a parabolic equation precisely if the immersion is spacelike. Further, if the ambient space is a product manifold $M\times\Rb$, 
the flow is given entirely in terms of a function
$u:M\times[0,T)\to\Rb$, subject to the equation
\[
	\partial_t u = \sqrt{1 - |\nabla u|_{\sigma}^2} \, \div_{\sigma}\left( \frac{\prescript{\sigma}{}\nabla u}{\sqrt{1-|\nabla u|_{\sigma}^2}} \right) \,,
\]
where $\sigma$ is the metric on $M$ and the norms and derivative are taken with respect to $\sigma$.
 Ecker and Huisken considered the mean curvature flow in cosmological spacetimes satisfying the timelike convergence condition to find hypersurface of prescribed mean curvature \cite{EH91b}. However, they had to assume a structural monotonicity condition for the prescribed mean curvature.
  The timelike convergence condition and the structural monotonicity condition were later on removed by Gerhardt \cite{Ger00}. Long-time existence of the mean curvature flow in Minkowski space was shown in \cite{Ecker97}. In \cite{Ecker93}, Ecker considered the mean curvature flow in asymptotically flat manifolds and proved long-time existence and convergence to a maximal hypersurface, provided that the spacetime satisfies a weak energy condition.

In this paper, we are considering the asymptotic behaviour of solutions of the mean curvature flow in asymptotically flat manifolds. In contrast to \cite{Ecker93}, we do not need any energy condition on the spacetime but we assume that the spacetime is of product type. Our main result is as follows:
\begin{theorem}
	Let $(M^n,\sigma)$, $n\neq2$ be an asymptotically flat Riemannian manifold and let $M\times \R$ be equipped with the Lorentzian product metric
	\[
		\sigma - (\d x^0)^2 \,.
	\]
	Furthermore, let $F_0\subset M\times\R$ be a hypersurface which is the graph of a function $u_0\in C^{0,1}(M)$ with Lipschitz constant smaller than $1-\varepsilon$ for some $\varepsilon>0$
	and such that $u_0(x)\to s$ for some $s\in\R$ as $x\to\infty$. Then there exists a solution $F_t$ to the mean curvature flow which starts at $F_0$ and exists for all times $t>0$. Further, $F_t$ is the graph of a function $u_t:M\to\R$ for each $t$ and $u_t$ converges to $u_{\infty}(x)=s$, uniformly in all derivatives as $t\to \infty$.
\end{theorem}
In this paper, we prove the theorem for manifolds with one Euclidean end but the proof can easily be generalised to the case of multiple ends. Similarly, the result also holds for manifolds which are asymptotically locally Euclidean (ALE).

Our solution of the mean curvature flow is constructed as the limit of solutions of Dirichlet problems on large balls. In spatial dimensions $n\geq3$, we construct rotationally symmetric stationary solutions of the mean curvature flow in Minkowski space. In asymptotic coordinates, these solutions serve as barriers in order to guarantee that the limit solution stays asymptotic to $M\times\{s\}$ for positive times, i.\,e.\ the solution can not lift off at spatial infinity in the sense that the function $|u_t-s|$ becomes eventually bigger than some positive constant on any compact subset. The behaviour of the hypersurfaces at spatial infinity has been raised as an open point in \cite[Remark 3.2]{Ecker93}.

Furthermore, we use these barriers at infinity and a carefully chosen quantity that controls the solution in the interior to prove uniform $C^1$-bounds from which we obtain global existence. The convergence behaviour at $t\to\infty$ follows from a combination of maximum principle arguments.

We have to exclude the case $n=2$ because we can neither construct nice barriers at infinity nor use the simple structure of the evolution equation as in the case $n=1$. Therefore, we get an unsatisfactory gap in the main theorem. However, we conjecture that the assertion of our main result also holds in dimension $n=2$.
\begin{corollary}
	\label{MainCor}
	Let $(M\times N,\sigma+\eta)$ be the Riemannian product of an asymptotically flat manifold $(M^n,\sigma)$, $n\neq2$, and a Riemannian manifold $(N,\eta)$ of bounded geometry (i.e.\ with positive injectivity radius, bounded curvature and boundedness of any derivative of curvature) and let $M\times N\times \R$ be equipped with the Lorentzian product metric
	\[
		\sigma + \eta - (\d x^0)^2\,.
	\]
	Furthermore, let $F_0\subset M\times N\times\R$ be a hypersurface which is given by the graph of a bounded function $u_0\in C^{0,1}(M\times N)$ with Lipschitz constant at most $1-\varepsilon$ for some $\varepsilon>0$
	and such that $u_0(x,y)\to s$ uniformly in $y\in N$ for some $s\in\R$ as $x\to\infty$. Then there exists a solution $F_t$ to the mean curvature flow which starts at $F_0$ and exists for all times $t>0$. Further, $F_t$ is the graph of a function $u_t:M\times N\to\R$ for each $t$ and $u_t$ converges to $u_{\infty}(x)=s$, uniformly in all derivatives as $t\to \infty$.
\end{corollary}
The corollary is a consequence of our main result: Pick an initial function $v_0:M\to \R$ 
satisfying the assumptions of the theorem and that in addition fulfills the inequality $|u_0(x,y)-s|\leq v_0(x)$ for all $(x,y)\in M\times N$. Then the mean curvature flow in $M\times \R$ starting at $\graph(v_0)$ (represented by the function $v_t$) can be extended to a mean curvature flow on $M\times N\times\R$ (by letting $v_t$ be constant in the $N$-direction).
By adding some $\varepsilon>0$ and applying a point-picking maximum principle we see that $v_t$
 serves as a barrier for $u_t$.
\begin{remark}If $N=\R^{n-1}$ (equipped with the flat metric) in Corollary \ref{MainCor}, one obtains the main theorem under weaker assumptions in the case of the Minkowski space:
	\begin{enumerate}[label=(\roman*)]
		\item If $M$ in the main theorem is $\R^n$ with the flat metric, the assertion also holds for $u_0\in C^{0,1}(\R^n)$ that does not converge to $s\in\R$ in all directions but is bounded and satisfies $u_0(x_1,\ldots,x_n)\to s$ uniformly in $x_2,\ldots,x_n$ as $|x_1|\to\infty$. 
		\item Suppose $F_0$ is a spacelike hypersurface that converges uniformly in one direction to an arbitrary spacelike hyperplane $\Sigma\subset \R^{n,1}$ in the following sense: There is an isometry $A\in \mathrm{Iso}(\R^{n,1})$ such that $A(\Sigma)=\R^{n}\times\left\{s\right\}$ and $A(F_0)=\graph(u_0)$ where $u_0$ satisfies the assumptions of (i). Then the mean curvature flow starting at $F_0$ exists for all times and converges to $\Sigma$ uniformly in all derivatives.
	\end{enumerate}
\end{remark}
From an analytical viewpoint, it would be very interesting to relax the condition on the Lipschitz constant of $u_0$ such that we can allow $|\nabla u_0|\to 1$ at infinity. Moreover, it would also be interesting to generalise the result from product spacetimes to stationary spacetimes. Both issues are planned to be attacked in the future.

This paper is organised as follows. In Section \ref{preliminaries}, we introduce notation, conventions and setup. In Section \ref{BarrierConstruction}, we construct rotationally symmetric barriers at infinity. In Section \ref{Interpolation}, a general procedure to interpolate between an arbitrary initial data on an asymptotically flat manifold and zero initial data on flat space is elaborated. In Section \ref{noliftoff}, we show that the mean curvature flow can not lift off at infinity due to our assumptions.
In Sections \ref{bgestimates} and \ref{destimates}, we 
prove gradient estimates for the Dirichlet problem at the boundary and in the interior, respectively. In Section \ref{mainresults}, we conclude with the proof of the main theorem.
\newline
\newline
\noindent \textbf{Acknowledgements.} This project was initiated at a winter school about geometric evolution equations in Konstanz in February 2018.
The authors thank the priority programme 2026 \textit{Geometry at Infinity}, funded by the \textit{Deutsche Forschungsgemeinschaft}, for financial support and for providing an excellent platform for joint research.
Klaus Kr\"oncke and Felix Lubbe have carefully connected our different contributions and prepared the present paper. The other authors are very grateful for these diligent efforts. 

\newpage
\section{Preliminaries}\label{preliminaries}

\subsection{Notation and conventions}

Let $(M,\s)$ be an $n$-dimensional Riemannian manifold. We consider the Lorentzian product manifold
\[
	\bigl(M\times\Rb, h\bigr)\,, \qquad \text{where} \quad h\coloneqq \sigma- \bigl(\d x^0\bigr)^2 \,,
\]
and where $x^0$ is the coordinate on $\Rb$. Locally (i.\,e.\ with respect to local coordinates $\{x^k\}$), 
the expression for the metric is given by
\[
	h = \sigma_{ij} \d x^i \d x^j - (\d x^0)^2 \,,
\]
where we use the Einstein summation convention and
the indices $i,j,k,\dots$ will always run from $1$ to $n\coloneqq \dim M$. Associated to the metric $h$ we have
its Levi-Civita connection $\prescript{h}{}\nabla$, which splits according to the product structure as
\[
	\prescript{h}{}\n = \prescript{\s}{}\n \oplus \frac{\partial}{\partial x_0} \,.
\]
Here, $\prescript{\s}{}\n$ is the Levi-Civita connection with respect to $\s$, with Christoffel symbols given
by $\ChrM_{ij}^k$.

A smooth function $u\colon M\to\Rb$ naturally defines an embedding
into the product manifold $M\times\Rb$ via
\[
	F\colon M\to M\times\Rb\,, \qquad F(x)\coloneqq \bigl( x, u(x) \bigr) \,.
\]
Then the graph of the map $u$ is the image of $F$,
\[
	\graph u \coloneqq F(M) \coloneqq \bigl\{ \bigl( x, u(x) \bigr) : x\in M \bigr\} \subset M\times\Rb \,.
\]
The induced metric $g$ on $M$ is defined as
\[
	g \coloneqq F^*h 
\]
and is locally given by
\[
	g_{ij} = (F^*h)_{ij} = \sigma_{ij} - u_i u_j \,,
\]
where $u_j \coloneqq \frac{\partial u}{\partial x_j}$ denotes the derivatives of the function $u$
with respect to the local coordinates $\{x^k\}$. The
Levi-Civita connection of $g$ will be denoted by $\n$.

For the Ricci curvatures associated to the metrics $h$ and $\s$, we will write $\Ric_h$ and $\Ric_{\s}$,
respectively. \\

The different norms induced by the metrics will be indicated by an index. For example,
we write $|\cdot|_g$ for the norm induced by the metric $g$. The gradient associated 
to a metric will be denoted by the same symbol as the corresponding connection
(e.\,g.\ $(\prescript{\sigma}{}\n u)^i = \sigma^{ij} u_j$).
Note that when writing the norm of a derivative, we will use
the same metric for the connection and for the norm (unless otherwise stated), e.\,g.\ $|\n u|_{\s} \coloneqq \abs{\prescript{\s}{}\n u}_{\s}$. \\

Throughout, we will often use the notation ``$x \lesssim y$'', which means ``there is a universal constant $C=C(n)>0$ such that $x\le Cy$'' and ``$x\simeq y$'' means ``$x\lesssim y$ and $y\lesssim x$''.\\

Let us endow $\R^n$ with its Euclidean metric $\delta_{ij}$.
For the following definition, 
norms $|\cdot|$, raising and lowering of indices, and differentiation $\nabla$ are taken with respect to the Euclidean metric, unless otherwise specified.
Moreover, for a function $f$ we denote by $f_{i}$ or, respectively, $f_{ij}$ derivatives with respect to the Euclidean metric, i.\,e.\ partial derivatives in Cartesian coordinates.

\begin{definition}
	A Riemannian manifold $(M,\sigma)$ is called \emph{asymptotically flat} if there exists a compact set $K\subset M$, a Euclidean ball 
	$B_{R_0}(0)\subset\r^n$ such that $M\setminus K$ is diffeomorphic to $\r^n\setminus B_{R_0}(0)$ and there is a bounded function $\omega \colon (R_0,\infty)\to \r_+$ 
	with $\omega(\xi)\to 0$ as $\xi\to \infty$ such that for all $r\geq R_0$,
	\begin{align}
		\left|\sigma_{ij}-\delta_{ij}\right| & \lesssim \omega(r),  \label{eq asym Eucl 1}\\
		\left|\nabla \left(\sigma_{ij}-\delta_{ij}\right)\right|  =  \left|\nabla \sigma_{ij}\right| & \lesssim \frac{\omega(r)}{r},\label{eq asym Eucl 2}
	\end{align}
	where $r$ is the Euclidean distance to the origin and the diffeomorphism is suppressed in the formul\ae.
\end{definition}

\begin{remark}
	\begin{enumerate}[leftmargin=1cm,label=(\roman*)]
		\item Note that asymptotic flatness implies
			\begin{align}
				\left|\sigma^{ij}-\delta^{ij}\right| & \lesssim \omega(r), \label{eq asym Eucl 3}\\
				\left| \ChrM^k_{ij} \right| = \left| \tfrac12 \sigma^{kl}(\sigma_{li,j}+\sigma_{lj,i}-\sigma_{ij,l})\right| & \lesssim \frac{\omega(r)}{r}. \label{eq asym Christoffel}
			\end{align}
		\item We use a weaker definition of asymptotically flatness than most other papers where one often assumes that $\omega(r)=r^{-\tau}$ for some $\tau>0$ (e.g.\ in \cite{Ecker93}). In this case, one calls $(M,\sigma)$ asymptotically flat of order $\tau$.
	\end{enumerate}
\end{remark}

The diffeomorphism $\varphi\colon M\setminus K\to \r^n\setminus B_{R_0}(0)$ is also called an asymptotic coordinate system and $\varphi=(x^1,\ldots,x^n)$ are called asymptotic coordinates.
 The Riemannian distance function of $\sigma$ will be denoted by $d$ and metric balls of radius $R$ around $x\in M$ will be denoted by $B_R(x)$. If $R\geq R_0$, we use the notation
\begin{align*}K_R\coloneqq M\setminus (\varphi^{-1}(\r^n\setminus B_{R}(0))).
\end{align*}
We call $M\setminus K_R=\varphi^{-1}\bigl(\r^n\setminus B_{R}(0)\bigr)$ an asymptotic coordinate neighbourhood.

Finally, we call $u\in C^{0,1}(M)$ uniformly spacelike, 
if there exists an $\varepsilon>0$, such that the Lipschitz constant (with respect to $\sigma$) is at most $ 1-\varepsilon$. For $u\in C^1(M)$, this is equivalent to $|\nabla u|_{\sigma}\leq1-\varepsilon$. The notion of uniform spacelikeness will become clearer in the next subsection.

\subsection{Mean curvature flow of spacelike hypersurfaces in static spacetimes}
\label{sec:MCF11}

Let $u\colon M\times[0,T)\to\Rb$ be a family of smooth functions,
where we set $u_t(x)\coloneqq u(x,t)$.
We say that the family $u$ evolves under the mean curvature
flow, if the family of graphs
\[
	\graph u_t \coloneqq \bigl\{ (x,u(x,t)) : x \in M \bigr\} \subset M\times\Rb
\]
satisfies the mean curvature flow equation, that is, there exists another family of maps $F\colon M\times [0,T)\to M\times \R$ such that
\begin{equation}
	\label{eq:MCF}
	\begin{cases}
		\partial_t F_t(x) = \Hv(x,t) \quad \forall x\in M \,, \\
		F_0(x) = (x,u_0(x)) \,,
	\end{cases}
\end{equation}
and $\graph u_t=F_t(M)$. Here, $\Hv(x,t)$ denotes the mean curvature vector of $F_t(M)$ at $F(x,t)$.
 To get a parabolic equation, we have to assume that the hypersurfaces $F_t(M)$ are spacelike. This is the case if and only if $|\nabla u|_{\sigma}<1$ everywhere. Under these assumptions, the mean curvature flow system may equivalently be written entirely in terms
of $u$ as
\begin{equation}
	\label{eq:gMCF}
	\partial_t u = \sqrt{1-| \nabla u |_{\sigma} ^2} \, \div_{\sigma} \left( \frac{\prescript{\s}{}\nabla u}{\sqrt{1-| \nabla u|_{\sigma}^2}} \right) \,.
\end{equation}
Here, 
$\div_{\s}$ denotes the divergence 
with respect to the metric $\s$.
Using that the inverse metric is given by
\[
	g^{ij}(\sigma,\nabla u) =\sigma^{ij} + \frac{\sigma^{ik} \sigma^{jl} u_k u_l}{1-\sigma^{kl}u_k u_l} \,,
\]
we can also write
\begin{equation} \label{eq MCF}
	\partial_t u=g^{ij}(\sigma,\nabla u) \, \prescript{\s}{} \nabla^2_{ij} u \,.
\end{equation}

\begin{lemma}
	\label{lem:evu}
	The evolution equation for $u$ is given by
	\[
		\left( \frac{d}{dt} - \Delta \right) u = 0 \,,
	\]
	where $\frac{d}{dt}$ denotes the total time derivative and
	$\Delta$ is the Laplacian with respect to the induced metric $\gind$.
\end{lemma}

\begin{proof}
	In \cite[Lemma 3.5]{Ger00} (using the notation of this paper),
we have the formula
\[
\left( \frac{d}{dt} - \Delta \right) u = e^{-\psi}v^{-1}f-e^{-\psi}g^{ij}\bar{h}_{ij}+\bar{\Gamma}_{00}^{0}|D u|^2+2\bar{\Gamma}_{0i}^0u^i,
\]	
for the mean curvature flow
\[
\partial_t F(x,t) = (\Hv-f\nu)(x,t)
\]
with a prescribed function $f$ in a general Lorentzian manifold. In our setting, $f=0$. The appearing Christoffel symbols vanish, because $M\times \R$ is equipped with a product metric. For the same reason, the second fundamental form $\bar{h}$ of the hypersurfaces $M\times\left\{s\right\}$ vanishes. Thus, the whole right hand side vanishes and we are left with the result of the lemma.
\end{proof}

Let us set
\begin{align}\label{definition_v}
	v \coloneqq \frac{1}{\sqrt{1-|\nabla u|^2_{\s}}} = - \left\< \nu, \frac{\partial}{\partial {x_0}} \right\> \,, \quad \text{where} \quad \nu \coloneqq \frac{\bigl(\prescript{\s}{}\nabla u, 1\bigr)}{\sqrt{1-|\nabla u|^2_{\s}}}
\end{align}
is the unit normal vector of the graph. 

\begin{lemma}\label{uniform_metrics}
	Assume that $\graph u_t$ is strictly spacelike.
	Then the relations
	\begin{align*}
		\frac{1}{v^2} \sigma &\le g \le \sigma \,, \\
		1 - \frac{1}{v^2} &= |\nabla u|^2_{\s} \le |\nabla u|^2 = v^2-1
	\end{align*}
	hold.
\end{lemma}

\begin{proof}
	The first relation follows by noting that for all vectors $w\in TM$ we have
	\[
		\frac{1}{v^2} = 1 - |\nabla u|^2_{\s} \,, \qquad |\nabla u|^2_{\s} \s(w,w) \ge (\partial_w u)^2
	\]
	and $g(w,w)=\sigma(w,w)-(\partial_wu)^2\le \sigma(w,w)$.
	For the second relation, we note that $\s^{-1}\le g^{-1}\le v^2 \s^{-1}$ and calculate
	\[
		|\n u|^2_{\s} = \s^{ij} u_i u_j \le g^{ij} u_i u_j = |\n u|^2 \,.
	\]
	For the remaining equality, we calculate
	\begin{align*}
		|\n u|^2 &= g^{ij} u_i u_j = \left( \sigma^{ij} + \frac{\sigma^{ik} \sigma^{jl} u_k u_l}{1-\sigma^{kl} u_k u_l} \right) u_i u_j \\
			&= |\n u|^2_{\s} + \frac{|\n u|^4_{\s}}{1-|\n u|_{\s}^2} = \frac{|\n u|_{\s}^2}{1-|\n u|_{\s}^2} = v^2-1 \,. \qedhere
	\end{align*}
\end{proof}

\section{Construction of a barrier}\label{BarrierConstruction}

For this section we assume $n\ge 3$.
In the cases $n=1,2$, the graphical maximal surface described below is unbounded.
Therefore, these cases are slightly more complicated and will be discussed in separate remarks in later chapters.

We wish to find a rotationally symmetric positive function $b$ defined on $\R^n \setminus B_{r_0}(0)$ (where $\R^n$ is equipped with an asymptotically flat metric $\sigma$) for some $r_0>0$, such that $b(x)\gg 1$ for $|x|=r_0$, $b(x)\to 0$ for $|x|\to \infty$, and such that $b$ is a static supersolution of \eqref{eq MCF}, i.\,e.\
\begin{equation} \label{eq supersolution}
	g^{ij}(\sigma,\nabla b) \, \prescript{\s}{} \nabla^2_{ij} b \equiv \left(\sigma^{ij} + \frac{\sigma^{ik} \sigma^{jl} b_k b_l}{1-\sigma^{kl}b_k b_l} \right) \left(b_{ij}- \ChrM^k_{ij} b_k\right) \le 0.
\end{equation}
Let us first consider $\sigma_{ij}(x)=\delta_{ij}$ and the case of equality in \eqref{eq supersolution}.
That is, we consider the equation for maximal spacelike hypersurfaces in Minkowski space.
The function corresponding to $b$ in this case we denote by $\beta$.
We choose $\beta$ to be rotationally symmetric and by abuse of notation we write $\beta(x)\equiv \beta(r)$.
Also setting $\beta^i = \beta_k \delta^{ki}$, \eqref{eq supersolution} becomes
\begin{align}
	0& = \left(\delta^{ij}+ \frac{ \beta^i \beta^j}{1-\beta^k\beta_k} \right) \beta_{ij} \nonumber \\
		& = \left( \delta^{ij} + \frac{(\beta')^2}{1-(\beta')^2}\frac{x^i x^j}{r^2}\right) \left( \beta'' \frac{x_i x_j}{r^2} + \beta' \left(\frac{\delta_{ij}}{r}-\frac{x_i x_j}{r^3}\right)\right) \nonumber\\
		& = \frac{\beta''}{1-(\beta')^2} + \beta' \frac{n-1}{r}. \label{eq Mink max surf}
\end{align}
Rewriting this equation in the form
\[
	\frac{\beta''}{(1-\beta')(1+\beta')\beta'} = -\frac{n-1}{r}
\]
makes it easy to integrate it under the condition $-1<\beta'<0$.
We have for any $c>0$ a solution of \eqref{eq Mink max surf} of the form
\begin{equation}\label{eq beta'}
	\beta' = -\left(1+c\,r^{2n-2}\right)^{-1/2}.
\end{equation}
We remark that $\beta'$ is integrable at infinity only for $n\ge 3$ and a solution with $\beta(r)\to 0$ for $r\to \infty$ is available only in this case.

\begin{lemma}\label{barrierlemma}
	There exists $r_1>0$, depending only on $\sigma$, such that for any $r_0 \ge r_1$ there exists a rotationally symmetric function $b\colon \R^n\setminus B_{r_0}(0) \to \R$ with the following properties:
	\begin{enumerate}[leftmargin=1cm,label=(\roman*)]
		\item $b(x) \simeq r_0^{n-\frac32}\, |x|^{-(n-\frac52)}$ and
		\item $g^{ij}(\sigma,\nabla b) \prescript{\s}{}\nabla^2_{ij}b \le 0$, i.\,e.\ $b$ is a static supersolution. 
	\end{enumerate}
\end{lemma}
\begin{proof}In order to construct a supersolution, we are going to tweak the exponent of $r$ in \eqref{eq beta'} to make $\beta$ into a proper supersolution $b$ strong enough to survive a change from the Euclidean metric to $\sigma$.
For $r_0>0$ (to be chosen later) we consider the rotationally symmetric function $b(r)$ given on $\R^n\setminus B_{r_0}(0)$ by the conditions $b(r)\to 0$ for $r\to \infty$ and
\begin{equation}
	b' = -\left(1+\left(\frac{r}{r_0}\right)^{2n-3}\right)^{-1/2}.
\end{equation}
The second derivative is given by
\begin{equation}
	b'' = \left(n-\tfrac32\right)\frac{\frac{1}{r_0}\left(\frac{r}{r_0}\right)^{2n-4}}{\left( 1+ \left(\frac{r}{r_0}\right)^{2n-3}\right)^{3/2}}.
\end{equation}
We will also need
\begin{equation}
	1-(b')^2 = 1- \frac{1}{1+\left(\frac{r}{r_0}\right)^{2n-3}} = \frac{\left(\frac{r}{r_0}\right)^{2n-3}}{1+\left(\frac{r}{r_0}\right)^{2n-3}}.
\end{equation}
Now, instead of fulfilling \eqref{eq Mink max surf}, $b$ is a strict supersolution:
\begin{align}
	\left(\delta^{ij}+ \frac{ b^i b^j}{1-b^kb_k} \right) b_{ij}& =  \frac{b''}{1-(b')^2} + b' \frac{n-1}{r} \nonumber \\
		&= (n-\tfrac32)  \frac{r^{-1}}{\left(1+\left(\frac{r}{r_0}\right)^{2n-3}\right)^{1/2}} +(n-1)\frac{b'}{r} \nonumber \\
		&= \frac12 \frac{b'}{r} \simeq  -  r^{-1}  \left(\frac{r}{r_0}\right)^{-n+\frac32} \,. \label{eq strict supersolution}
\end{align}
We shall see that the left-hand sides of \eqref{eq supersolution} and \eqref{eq strict supersolution} differ by terms of lower order than the right-hand side of \eqref{eq strict supersolution}.
In this way we will obtain \eqref{eq supersolution}.
Before we continue let us remark the following:
\begin{align}
	|\nabla b| & \simeq \left(\frac{r}{r_0}\right)^{-n+\frac32}, \label{eq asymptotics Db}\\
	|\nabla^2 b| & \simeq  r^{-1}\left(\frac{r}{r_0}\right)^{-n+\frac32}. \label{eq asymptotics D^2b}
\end{align}
Furthermore we note $1-\delta^{kl}b_k b_l \ge \frac12$ and by assuming that $r_0$ is sufficiently large we can infer by \eqref{eq asym Eucl 1} that $1-\sigma^{kl}b_kb_l \ge \frac13$.
It allows us to estimate in the following way using \eqref{eq asym Eucl 3} and $|\nabla b|\le 1$:
\begin{equation} \label{eq 1. error term}
	\left| \frac{1}{1-\sigma^{kl}b_kb_l} - \frac{1}{1-\delta^{kl}b_kb_l}\right|
		= \left| \frac{(\sigma^{kl}-\delta^{kl})b_kb_l}{(1-\sigma^{kl}b_kb_l)(1-\delta^{kl}b_kb_l)} \right|
		\lesssim \omega(r).
\end{equation}
It is now straightforward to estimate the error term from $g^{ij}$:
\begin{equation} \label{eq metric comparison}
	\begin{split}
		&\quad \left| g^{ij}(\sigma,\nabla b)-g^{ij}(\delta,\nabla b) \right| \\
		& = \left| \sigma^{ij} + \frac{\sigma^{ik}\sigma^{jl} b_kb_l}{1-\sigma^{kl}b_kb_l}
			- \delta^{ij} - \frac{\delta^{ik}\delta^{jl} b_kb_l}{1-\delta^{kl}b_kb_l}\right|\\
		& \le \left|\sigma^{ij}-\delta^{ij}\right| 
			+ \left| \frac{\sigma^{ik}\sigma^{jl}-\delta^{ik}\delta^{jl}}{1-\sigma^{kl}b_kb_l} \right|
			+ \left| \delta^{ik}\delta^{jl} \right| \left| \frac{1}{1-\sigma^{kl}b_kb_l}-\frac{1}{1-\delta^{kl}b_kb_l}\right|\\
		& \lesssim \omega(r).
	\end{split}
\end{equation}
We can finally compare the operators for the different metrics:
\begin{equation} \label{eq operator comparison}
	\begin{split}
		& \quad \left| g^{ij}(\sigma,\nabla b) \prescript{\s}{}\nabla^2_{ij}b - g^{ij}(\delta,\nabla b) b_{ij} \right|\\
		& \le \left| g^{ij}(\sigma,\nabla b) - g^{ij}(\delta,\nabla b) \right| |b_{ij}|
			+ \left|g^{ij}(\sigma,\nabla b)\right| \left| \ChrM^k_{ij}b_k\right|\\
		& \lesssim \frac{\omega(r)}{r} \left(\frac{r}{r_0}\right)^{-n+\frac32}
	\end{split}
\end{equation}
by \eqref{eq metric comparison}, \eqref{eq asymptotics D^2b},  \eqref{eq asymptotics Db}, \eqref{eq asym Christoffel}, and the fact that $\left|g^{ij}(\sigma,\nabla b)\right|$ is bounded (by \eqref{eq asym Eucl 3} and $r_0\gg 1$).
Since by \eqref{eq strict supersolution}
\begin{equation}
	g^{ij}(\delta,\nabla b) b_{ij}= \left(\delta^{ij}+ \frac{ b^i b^j}{1-b^kb_k} \right) b_{ij}
		\simeq -   r^{-1}  \left(\frac{r}{r_0}\right)^{-n+\frac32} \,,
\end{equation}
we use \eqref{eq operator comparison} for $1\ll r_0 \le r$ to conclude the desired result \eqref{eq supersolution}, i.\,e.\ $b$ is a supersolution.
As we choose $b$ such that $b(r)\to 0$ for $r\to \infty$ we find
\begin{equation}
	b(r)
	= -\int\limits_{r}^{\infty} b'(s) \, \d s
	\simeq  \int\limits_{r}^{\infty} \left(\frac{s}{r_0}\right)^{-n+\frac32} \d s
	\simeq r_0^{n-\frac32}\, r^{-n+\frac52}
\end{equation}
which finishes the proof.
\end{proof}
For making use of the supersolution $b$ as a barrier, we need to make sure that a solution can not touch $b$ at $r=r_0$. 
But this is ensured by the lemma,
since $b$ can be made arbitrarily large at $r=r_0$ as we enlarge $r_0$.
\begin{corollary} \label{cor barrier at infinity}
	For any $h>0$ and for any $r_1>0$ there exists $r_0\ge r_1$, depending only on $\sigma$, and a rotationally symmetric function $b\colon \R^n\setminus B_{r_0}(0) \to \R$ with the following properties:
	\begin{enumerate}[leftmargin=1cm,label=(\roman*)]
		\item $b>0$,
		\item $b(x)\to 0$ for $|x|\to \infty$,
		\item $b(x)\ge h$ for $|x|=r_0$, and 
		\item $g^{ij}(\sigma,\nabla b) \prescript{\s}{}\nabla^2_{ij}b \le 0$, i.\,e.\ $b$ is a static supersolution. 
	\end{enumerate}
\end{corollary}

\begin{remark}\label{rem barrier at infinity}The following observations are crucial:
	\begin{enumerate}[leftmargin=1cm,label=(\roman*)]
		\item Because the operator only depends on derivatives, $b+\eps$ is a supersolution for any $\eps\in \R$.
		\item Because of $g^{ij}(\sigma, -\nabla b)= g^{ij}(\sigma, \nabla b)$, $-b$ is a subsolution and yields a barrier from below.
	\end{enumerate}
\end{remark}

\section{Interpolation of initial data}\label{Interpolation}
\begin{proposition}\label{interpolation_thm}
Let $(M,\s)$ be an asymptotically flat Riemannian manifold and $R_0>0$ so large that we have an asymptotic coordinate system
\begin{align*}
\varphi:M\setminus K\to \R^n\setminus B_{R_0}(0).
\end{align*} Let $u_0\in C^{0,1}(M)$ be uniformly spacelike with Lipschitz constant $1-\varepsilon$. Then for every  $R_2>R_1\geq R_0$, there exists a metric $\wt{\s}$ and a function $\wt{u}_0$ such that
	\begin{align*} 
		\wt{\s}_{ij} = \s_{ij} \quad &\text{and} \quad \wt{u}_0 = u_0 \qquad \text{on} \;\; K_{R_1},\\
		\wt{\s}_{ij} = \delta_{ij} \quad &\text{and} \quad \wt{u}_0 = 0 \qquad \text{on} \;\; M\setminus K_{R_2}
	\end{align*}
	and such that $\wt{u}_0$ is uniformly spacelike with respect to the constant $\varepsilon>0$ from above. \end{proposition}
\begin{proof}For the proof, we may assume that $u_0\in C^1(M)$ since the Lipschitz case follows immediately from an approximation argument.
	For notational convenience, set $S_1\coloneqq R_1$ and $S_4\coloneqq R_2$. Pick constants $S_2,S_3$ such that $S_1<S_2<S_3<S_4$
	and choose smooth cutoff functions $\psi_i:M\to [0,1]$, $i=1,2,3$ such that
	\begin{align*}
	\psi_i\equiv 1 \quad \text{ on }K_{S_i},\qquad 	\psi_i\equiv 0 \quad \text{ on }M\setminus K_{S_{i+1}}.
	\end{align*}
		Now fix $\lambda > 1$ and let 
	\begin{align*} 
	\wh{\s}_{ij} &= \psi_1 \s_{ij} + (1-\psi_1) \lambda \s_{ij} \,,  \\
	\wt{\s}_{ij} &= \psi_3 \wh{\s}_{ij} + (1-\psi_3) \delta_{ij} = \psi_3 (\psi_1 \s_{ij} + (1-\psi_1) \lambda \s_{ij}) + (1-\psi_3) \delta_{ij}
	\end{align*} 
		and
	\[ 
	\wt{u}_0 = \psi_2 u_0 \,. 
	\]
	Then we have $	\wt{\s}_{ij} = \s_{ij}$ and $\wt{u}_0 = u_0 $ on $K_{S_1}$
	and $\wt{\s}_{ij} = \delta_{ij} $ and $\wt{u}_0 = 0$ outside of $K_{S_4}$.
	Because $\wt{u}_0$ vanishes even outside of $K_{S_3}$, it remains to prove that  $\wt{u}_0$ is uniformly spacelike on $K_{S_3}\setminus K_{S_1}$. On $K_{S_2}\setminus K_{S_1}$, we have 
		\[ 
	\wt{u}_0 = u_0 \quad \text{and} \quad \wt{\s}_{ij} = \wh{\s}_{ij} \,. 
	\]
and since $\wh{\s}_{ij} = f \s_{ij}$ with a function $f \ge 1$, we have
\begin{align*} 
|\n \wt{u}_0|^2_{\wt{\sigma}} &= |\n u_0|^2_{\wh{\sigma}} =  \frac{1}{f} |\n u_0|^2_{\s}\le (1-\ve)^2
\end{align*}
in this region. Finally, let us check the region $K_{S_3}\setminus K_{S_2}$. There we have $	\wt{\s}_{ij} = \l \s_{ij}$ so that
	\begin{align*}
	|\n \wt{u}_0|^2_{\wt{\sigma}} = \frac{1}{\l} 	|\n \wt{u}_0|^2_{{\sigma}}\le \frac{2}{\l} \Bigl(u_0^2 |\n \psi_2|^2_{\s} + \psi_2^2|\n u_0|^2_{\s}\Bigr)\leq (1-\ve)^2 \,. 
	\end{align*}
	provided that $\l > 1$ was chosen large enough depending on $u_0,\varepsilon$ and $\psi_2$.
	This finishes the proof.
\end{proof}

\begin{remark}
	Note that in this construction the geometry of $\tilde{\s}$ can be chosen uniformly bounded (depending on $\sup_M |u_0|$, $\varepsilon$ and the difference $R_2-R_1$). Note further that $\tilde{u}_0(x)$ has the same sign as $u_0(x)$ for all $x\in M$ and that $|\tilde{u}_0|\leq |u_0|$. 
\end{remark}

\section{No lift-off}\label{noliftoff}

In the following (Lemma \ref{lem no lift-off for constructed} and Remark \ref{rem no lift-off}) we will show that our constructed solution has no lift-off behaviour at infinity.
Because our construction involves approximation by compact problems, the proof is an easy maximum principle argument in this case.
On the other hand, excluding lift-off behavior directly for any given solution needs an additional argument to compensate for the non-compactness and is therefore treated afterwards (Theorem \ref{thm no lift-off}).
\begin{lemma} \label{lem no lift-off for constructed}
Let $\sigma$ be an asymptotically flat metric on $\R^n$ $(n\ge 3)$ and $u_0\in C^{0,1}(\R^n)$ a function which is uniformly spacelike and satisfies $|u_0(x)|\to 0$ as $|x|\to \infty$. 	Let $R>0$ and  $u_{0,R}$ be a function constructed from $u_0$ according to Proposition \ref{interpolation_thm} that vanishes outside of $B_R(0)$ and is equal to $u_0$ on a smaller ball. 
Finally, let $T\in (0,\infty]$ 
	and assume that $u_R\in C^{2;1}\big(B_R(0)\times(0,T)\big)\cap C^0\big(\overline{B_R(0)}\times [0,T]\big)$ is a solution of the initial boundary value problem
	\begin{equation} \label{eq ibvp}
		\begin{cases}
			\partial_t u_R(x,t)=g^{ij}(\sigma,\nabla u_R)  \prescript{\s}{}\nabla^2_{ij} u_R(x,t) & \text{for } (x,t) \in B_R(0)\times (0,T),\\
			u_R(x,0) = u_{0,R}(x) & \text{for } x\in B_R(0),\\
			u_R(x,t) = 0 & \text{for } (x,t) \in \partial B_R(0)\times [0,T],
		\end{cases}
	\end{equation}
	which is extended by zero for $|x|>R$. Then for every $\eps>0$,
 there is a function $b_\eps \colon \R^n\to \R_+$, depending only on  $\sigma$ (through \eqref{eq asym Eucl 1} and \eqref{eq asym Eucl 2}), $\eps$, and $u_0$, such that
	\begin{enumerate}[leftmargin=1cm,label=(\roman*)]
		\item $|u_R(x,t)| \le b_\eps(x)$ for $(x,t)\in \R^n\times (0,T)$ and 
		\item $b_\eps(x)\to \eps \text{ for } |x| \to \infty.$
	\end{enumerate}
\end{lemma}

\begin{remark} \label{rem no lift-off}
Let $b(x)\coloneqq \inf_{\eps>0}b_{\eps}(x)$, which depends on $\sigma$ and $u_0$. Then
	\begin{enumerate}[leftmargin=1cm,label=(\roman*)]
		\item $|u_R(x,t)| \le b(x)$ for $(x,t)\in \R^n\times (0,T)$ and
		\item $b(x)\to 0 \text{ for } |x| \to \infty.$
	\end{enumerate}
\end{remark}

\begin{proof}[Proof of Lemma \ref{lem no lift-off for constructed}]
	The maximum principle implies $\sup_{\R^n\times [0,T)} u_R = \sup_{\R^n} u_{0,R}$. For $\eps>0$, there exists $r_1>0$ such that $|u_0(x)|<\eps$ for $|x|\ge r_1(\eps)$.
	Using $h\coloneqq \sup_{\R^n\times [0,T)}u_R$, Corollary \ref{cor barrier at infinity} gives an $r_0\ge r_1$ and a function $b\colon \R^n\setminus B_{r_0}(0)\to\R_+$ with the properties listed there. Instead of that we will use $b_\eps\coloneqq b+\eps$ (cf. Remark \ref{rem barrier at infinity}).
	Finally extend $b_\eps$ to $B_{r_0}(0)$ by $h+\varepsilon$.
		The asserted property (ii) is readily seen.
	
	For $|x|< r_0$, we have the inequality $u_R(x,t) \le h\le b_\eps(x)$ and for $|x|\ge R$, we have $|u_R(x,t)|=0<b_\eps(x)$, so it only remains to prove assertion (i) in the region $B_R(0)\setminus B_{r_0}(0)$ (if non-empty).
	To achieve this we use the comparison principle.
	Because of $r_0 \ge r_1$ we have by assumption $b_\eps(x)>\eps \ge u_R(x,0)$ in $B_R(0)\setminus B_{r_0}(0)$.
	By the definition of $h$ we have $b_\eps(x)\ge h+\eps > u_R(x,t)$ for $|x|=r_0$ and also  $b_\eps(x)>0=u_R(x,t)$ for $|x|=R$.
	So we have $u_R(x,t)< b_{\varepsilon}(x)$ on the parabolic boundary of $\bigl(B_R(0)\setminus B_{r_0}(0)\bigr) \times (0,T)$.
	Note that $b_\eps$ is a supersolution in the viscosity sense as a minimum of two supersolutions. Thus, the comparison principle implies $u_R(x,t)<b_\eps(x)$ on $B_R(0)\setminus B_{r_0}(0)$.
	
	An analogous argument shows $u_R(x,t)>-b_\eps(x)$ on $B_R(0)\setminus B_{r_0}(0)$ (cf.\ Remark \ref{rem barrier at infinity}).
\end{proof}
\begin{remark} \label{rem no lift-off_AEcase}
	The assertions of Lemma \ref{lem no lift-off for constructed}, and therefore of Remark \ref{rem no lift-off}, also hold when replacing $(\R^n,\sigma)$ by an asymptotically flat manifold and $B_R(0)$ by $K_R$. In this case, one uses constant barriers in the interior of the asymptotic neighbourhood.
\end{remark}
Later on, we will show that the $u_R$ converge (for $R\to\infty$) to a solution of the initial value problem on the whole manifold that preserves the asymptotic behaviour. In the following, we will show that this behaviour is preserved for any solution of the initial value problem. Uniqueness of such a solution is shown by using the comparison principle.
\begin{theorem} \label{thm no lift-off}
	Let $T\in (0,\infty]$ and let $\sigma$ be an asymptotically flat metric on $\R^n$ $(n\ge 3)$. Suppose $u_0\in C^{0,1}(\R^n)$ is uniformly spacelike and satisfies
	\begin{equation} \label{eq thm no lift-off 1}
		u_0(x)\to 0 \quad \text{for} \quad |x|\to \infty.
	\end{equation}
	Let $u\in C^{2;1}\bigl(\R^n\times (0,T)\bigr)\cap C^0\bigl(\R^n\times [0,T)\bigr)$ be a solution of \eqref{eq MCF} with $u(\cdot,0)=u_0$. 
	Further suppose that the solution is uniformly spacelike, i.\,e.\ there exists $\mu>0$ such that $|\nabla u|_\sigma^2\le 1- \mu$.
	Then for any $\varepsilon>0$ and any $t^*\in(0,T)$, there exists $R>0$, such that
	\begin{equation}
		\label{eq thm no lift-off 2}
		|u(x,t)| \le \varepsilon \qquad \text{for all} \quad |x|\ge R \quad \text{and} \quad t\in(0,t^*] \,.
	\end{equation}
	Moreover, this convergence is uniform in $t$ and in $u$ in the following sense.
	There is a function $b\colon \mathbb{R}^n\to \R_+$, dependent only on $u_0$ and $\sigma$, such that
	\begin{enumerate}[leftmargin=1cm,label=(\roman*)]
		\item $|u(x,t)|\le b(x)$ for $(x,t)\in \mathbb{R}^n\times (0,T)$ and
		\item $b(x)\to 0$ for $|x|\to \infty$.
	\end{enumerate}
\end{theorem}

\begin{remark}	
	The exact dependence on $u_0$ in Theorem \ref{thm no lift-off} is given through 
	$\sup_{\mathbb{R}^n} u_0$ and the function $r_1(\eps) \coloneqq\sup \{ |x| : |u_0(x)|>\eps \}$.
	The dependence on $\sigma$ is given through \eqref{eq asym Eucl 1} and \eqref{eq asym Eucl 2}.
\end{remark}

\begin{remark}
	The barrier construction explained in the proof below also works for dimensions $n=1,2$. As a consequence, 
	\eqref{eq thm no lift-off 2} also holds in these dimensions but we do not get (i) and (ii) in Theorem \ref{thm no lift-off} 
	because Lemma  \ref{lem no lift-off for constructed} and Remark \ref{rem no lift-off} cannot be used in these cases.
\end{remark}

\begin{proof}[Proof of Theorem \ref{thm no lift-off}]
	The essential part of the proof is to establish \eqref{eq thm no lift-off 2}. Once this is
	shown, we consider $b_{\varepsilon}$ as defined in the proof of Lemma \ref{lem no lift-off for constructed}.
	We deduce from \eqref{eq thm no lift-off 2} that for each $t^*\in(0,T)$ and
	each $\zeta$ with $0<\zeta<\varepsilon$ there exists $R\ge 1$, such that
	$b_{\varepsilon} - u \ge \zeta$ in $\bigl(\Rb^n \setminus B_R(0) \bigr)\times[0,t^*)$.
	Hence, we can apply the maximum principle, albeit $\Rb^n$ is not compact, and
	deduce that $b_{\varepsilon}\ge u$. Then we take the infimum over $\varepsilon>0$
	as in Remark \ref{rem no lift-off} and arrive at the second assertion about the barrier function
	
	
	Now we turn to the technical details.
	Firstly, we are assuming that we work far out where $\sigma$ is close to the Euclidean metric.
	To begin with, by choosing $\mu>0$,
	we may assume that
	\begin{equation}
		|\nabla u|^2\le 1-\mu \label{eq spacelikeness bound}
	\end{equation}
	holds in the region we consider.	
	Now we define the barrier function.
	Let $\alpha>0$. For $t_0<-1$ and $x_0\in \mathbb{R}^n$ we define
	\begin{multline*}
		\wh{b}(x,t)\equiv \wh{b}_{x_0,t_0}(x,t) \coloneqq \sqrt{2n(t-t_0)+|x-x_0|^2} + \alpha \, t, \\ (x,t)\in B_{\rho}(x_0)\times [0,-t_0],
	\end{multline*}
	where $\rho \equiv \rho(t_0) \equiv \rho(t_0,\mu) \coloneqq \sqrt{\frac{2-\mu}{\mu}4n(-t_0)}$.
	We are now going to prove that the function $\wh{b}$ satisfies
		\begin{align}
			\partial_t \wh{b} & = \frac{n}{\sqrt{2n(t-t_0)+|x-x_0|^2}} + \alpha,\label{eq lc derivative 1}\\
			\wh{b}_i & = \frac{(x-x_0)_i}{\sqrt{2n(t-t_0)+|x-x_0|^2}},\label{eq lc derivative 2}\\
			\wh{b}_{ij} & = \frac{1}{\sqrt{2n(t-t_0)+|x-x_0|^2}} \left( \delta_{ij} - \frac{ (x-x_0)_i (x-x_0)_j}{2n(t-t_0)+|x-x_0|^2}\right),\label{eq lc derivative 3}\\
			\partial_t \wh{b}&-g^{ij}(\delta,\nabla \wh{b}) \wh{b}_{ij}  = \alpha, \label{eq lc solution}\\
			1&-|\nabla \wh{b}|^2  \ge \frac{\mu}{4}, \label{eq lc gradbound}  \\
			\wh{b}_i&\frac{(x-x_0)^i}{|x-x_0|}  \ge \sqrt{1-\frac{\mu}{2}} \quad \text{for } |x-x_0|=\rho. \label{eq lc boundary}
		\end{align}
	Direct calculations yield equations \eqref{eq lc derivative 1}, \eqref{eq lc derivative 2}, and \eqref{eq lc derivative 3}.
		For the ease of notation we may ignore $\alpha$ and $x_0$ (by setting them equal to zero), for it is easy to 
		incorporate these afterwards.
		Now we address \eqref{eq lc solution}.
		We have $\wh{b}_i = \frac{x_i}{\wh{b}}$ and
		\begin{align*}
			g^{ij}(\delta,\nabla \wh{b})\wh{b}_{ij}
			&= \left(\delta^{ij} + \frac{x^i x^j}{\wh{b}^2-|x|^2}\right)\left(\frac{\delta_{ij}}{\wh{b}}-\frac{x_i x_j}{\wh{b}^3}\right) \\
			& = \frac{n}{\wh{b}}-\frac{|x|^2}{\wh{b}^3}+\frac{|x|^2}{\wh{b}\,(\wh{b}^2-|x|^2)} -\frac{|x|^4}{\wh{b}^3(\wh{b}^2-|x|^2)}\\
			& = \frac{n}{\wh{b}}- \frac{(\wh{b}^2-|x|^2)|x|^2 - \wh{b}^2 |x|^2 + |x|^4}{\wh{b}^3(\wh{b}^2-|x|^2)} = \frac{n}{\wh{b}}\\
			& = \partial_t \wh{b} \,.
		\end{align*}
		The inequality \eqref{eq lc gradbound} is shown by
		\begin{equation}
			1-|\nabla \wh{b}|^2 = \frac{2n(t-t_0)}{2n(t-t_0)+|x|^2} \ge \frac{2n(-t_0)}{2n(-t_0)+\rho^2} = \frac{1}{1+\frac{2-\mu}{\mu}2} = \frac{\mu}{4-\mu} \ge \frac{\mu}{4}.
		\end{equation}
		Finally, we derive the estimate \eqref{eq lc boundary}: With $|x|=\rho$ and $t\le -t_0$, we get
		\begin{equation}
			\wh{b}_i\frac{x^i}{|x|} = \frac{|x|}{\wh{b}} = \sqrt{\frac{\rho^2}{2n(t-t_0)+\rho^2}}
				\ge \sqrt{\frac{\rho^2}{4n(-t_0)+\rho^2}}
				= \sqrt{\frac{\frac{2-\mu}{\mu}}{1+\frac{2-\mu}{\mu}}}
				= \sqrt{1-\frac{\mu}{2}}.
		\end{equation}
	The action of the parabolic operator on $\wh{b}$ can be estimated in the following way:
	\begin{align*}
		& \partial_t \wh{b} -g^{ij}(\sigma,\nabla \wh{b}) \prescript{\s}{}\nabla^2_{ij}\wh{b}\\
		& \quad \ge \partial_t \wh{b}-g^{ij}(\delta,\nabla \wh{b})\wh{b}_{ij} - \left| g^{ij}(\delta,\nabla \wh{b})\wh{b}_{ij} - g^{ij}(\sigma,\nabla \wh{b}) \prescript{\s}{}\nabla^2_{ij}\wh{b} \right|\\
		& \quad \ge \alpha - \left| g^{ij}(\sigma,\nabla \wh{b}) - g^{ij}(\delta,\nabla \wh{b}) \right| |\wh{b}_{ij}|
			- \left|g^{ij}(\sigma,\nabla \wh{b})\right| \left| \ChrM^k_{ij}\right| |\wh{b}_l| \,.
	\end{align*}
	We may use \eqref{eq lc gradbound}, the calculation of \eqref{eq metric comparison}, $|\wh{b}_{ij}|\lesssim 1$ for $-t_0\gg1$, and \eqref{eq asym Christoffel} to conclude
	\begin{equation} \label{eq lc supersolution}
		\partial_t \wh{b} -g^{ij}(\sigma,\nabla \wh{b}) \prescript{\s}{}\nabla^2_{ij}\wh{b} > 0
	\end{equation}
	if we are sufficiently far out, i.\,e.\ $|x_0|-\rho(t_0)$ is sufficiently large depending on $\alpha$, $\mu$, and $\sigma$ only.
	
	We have shown that $\wh{b}+h$ ($h\in \mathbb{R}$) is an upper barrier for $u$ on $B_\rho(x_0)\times [0,-t_0]$ ($|x_0|$ large) if $\wh{b}+h>u$ holds initially: 
	Because of \eqref{eq spacelikeness bound} and \eqref{eq lc boundary} there can not be a first contact point of $u$ and $\wh{b}$ at $\partial B_\rho(x_0)\times (0,-t_0)$. By \eqref{eq lc supersolution} and {\eqref{eq MCF}}, there can not be a first contact in the interior either. For more a more detailed exposition, see Proposition \ref{gradient_max_principle} below.
	
	To prove \eqref{eq thm no lift-off 2} we show that for any $t^*\in (0,T)$ and any $\varepsilon>0$ there is $R>0$ such that
	$u(x,t) < \varepsilon$
	for any $|x|>R$
	and any $t\in(0,t^*]$.
	So let $t^*\in (0,T)$ and $\varepsilon>0$ be given.
	By \eqref{eq thm no lift-off 1} there is $R_0$ such that $u(x,0)<\frac{\varepsilon}{2}$ for $|x|\ge R_0$.
	Let us choose $|t_0|$ so large and $\alpha$ so small that $\partial_t \wh{b}_{x_0,t_0}(x_0,t)<\frac{\varepsilon}{2t^*}$
	for all $t\in[0,t^*]$.
	After that we fix an arbitrary $x_0\in \mathbb{R}^n$ with $|x_0|>R_0+\rho(t_0)$ and $|x_0|$ being so large that \eqref{eq spacelikeness bound} and \eqref{eq lc supersolution} hold
	in $B_{\rho}(x_0)\times(0,-t_0)$. Recall also \eqref{eq lc boundary}.
	Then $\wh{b}_{x_0,t_0}+h$ with $h=-\wh{b}_{x_0,t_0}(x_0,0)+\frac{\varepsilon}{2}$ is an upper barrier for $u$ in $B_\rho(x_0)\times (0,-t_0)$. We may assume $t^*<-t_0$.
	In particular, we obtain for $t\in(0,t^*)$
	\begin{equation*}
		u(x_0,t)<\wh{b}_{x_0,t_0}(x_0,t)+h \le \frac{\varepsilon}{2}+t^* \sup_{\tau\in [0,t^*]} \partial_t \wh{b}_{x_0,t_0}(x_0,\tau) \le \varepsilon \,,
	\end{equation*}
	which establishes \eqref{eq thm no lift-off 2}.
\end{proof}
To prove a version of this theorem for asymptotically flat manifolds,
we apply a comparison principle for barriers with large
gradient at the boundary and solutions with bounded gradient.
 In \cite[Lemma 4.1]{CSS07}, a similar
technique has been used. 
\begin{proposition}\label{gradient_max_principle}
	Let $\Omega$ be an open bounded subset of a complete manifold $M$
	with $\partial\Omega\in C^1$ and outer unit normal $\nu$. Let
	$0<T\le\infty$. Let
	$u\in C^{2;1}(M\times(0,T))\cap C^0(M\times[0,T))$ be a solution to
	\eqref{eq MCF} in $\Omega\times(0,T)$. Let
	$b\in C^{2;1}\bigl(\overline\Omega\times[0,T)\bigr)$ fulfill 
	\[\partial_t b\ge g^{ij}(\sigma,\nabla b) \prescript{\sigma}{}\nabla^2_{ij} b\quad\text{in }\Omega\times(0,T).\] 
	Assume that
	\[u(\cdot,0)\le b(\cdot,0)\quad\text{in }\Omega\quad\text{and}\quad
	\sup\limits_{\Omega\times(0,T)}|\nabla u|
	<\inf\limits_{\partial\Omega\times(0,T)}\langle\nabla
	b,\nu\rangle.\] Then
	\[u\le b\quad\text{in }\Omega\times[0,T).\]
\end{proposition}
\begin{proof}
	By replacing $b$ by $b+\varepsilon(1+t)$ and later considering
	$\varepsilon\searrow0$, we may assume without loss of generality
	that $\partial_t b> g^{ij}(\sigma,\nabla b) \prescript{\s}{} \nabla^2_{ij}b$
	and $u(\cdot,0)<b(\cdot,0)$. It suffices to prove the result for any
	finite $S$ such that $0<S<T$. We exhaust $\Omega$ by compact sets
	$K_l$, $l\in\mathbb N$, such that $K_l\nearrow\Omega$ and
	$\partial K_l\to\partial\Omega$ in $C^1$ in a local graph
	representation. We may assume without loss of generality that $l$ is
	large enough so that
	\begin{equation}
	\label{Db large at bdry eq}
	\sup\limits_{\Omega\times(0,S]} |\nabla u|< \inf\limits_{\partial
		K_l\times(0,S]} \langle\nabla b,\nu\rangle,
	\end{equation}
	where $\nu$ is the outer unit normal to $\partial K_l$. We fix $l$
	and argue by contradiction. Hence, we may assume that there exists a
	first $(x_0,t_0)\in K_l\times(0,S]$, such that
	$u(x_0,t_0)=b(x_0,t_0)$. Now, we distinguish two cases:
	\begin{itemize}
		\item If $x_0\in \partial K_l$, we still have $b-u\ge 0$ in
		$K_l\times\{t_0\}$ with equality at $x_0$ and deduce that
		$\langle\nabla(b-u),\nu\rangle\le0$ at $(x_0,t_0)$. We get
		$\langle\nabla b,\nu\rangle\le|\nabla u|$ there. This contradicts
		\eqref{Db large at bdry eq}. 
		\item A first touching of $u$ and $b$ in the interior of $K_l$,
		however, is excluded by the maximum principle as we have assumed
		that $b$ is a strict supersolution.
	\end{itemize}
	We obtain $u<b$ in $K_l\times[0,S]$. Now, we let $l\to\infty$ and
	$S\nearrow T$.  The claim follows.
\end{proof}
\begin{corollary}
	Theorem \ref{thm no lift-off} holds for asymptotically flat $(M^n,\sigma)$  with $n\geq3$.
\end{corollary}
\begin{proof}
The only case where the proof slightly differs from the proof of Theorem \ref{thm no lift-off} is the second last paragraph:
In this case one chooses $|x_0|$ (in an asymptotic neighbourhood) and $\rho>0$ so large that $M\setminus B_{\rho}(x_0)$ lies in an asymptotic coordinate neighbourhood and such that $|\nabla b|>\sup_M |\n u_0|$ on $\partial B_{\rho}(x_0)$. Then, $v+h$ is an upper barrier for $u$ by Proposition \ref{gradient_max_principle}.
\end{proof}

\section{Boundary gradient estimates}\label{bgestimates}
Our goal in this section is to prove boundary gradient estimates for the modified Dirichlet problem for the mean curvature flow.

\begin{proposition}\label{boundary_gradient_thm}
Let $(M,\sigma)$ be an asymptotically flat manifold and $u_0\in C^{0,1}(\R^n)$ be a function which is uniformly spacelike and satisfies $|u_0(x)|\to 0$ as $|x|\to \infty$.
 Let $r_0$ (depending on $(M,\sigma)$ and $u_0$) be so large that $\R^n\setminus B_{r_0}(0)$ lies in the image of an asymptotic coordinate system and such that there exists a rotationally symmetric static supersolution $b_{r_0}$ of the mean curvature flow with $b_{r_0}=\sup u_0+1$ at $\partial K_{r_0}$.
 Let $R>\max\{2,r_0\}$ and let $u_{0,R}$ be modified initial data and $\sigma_{R}$ be a modified metric, such that
$u_{0,R}=0$ and $\sigma_{R}=\delta$ on $M\setminus K_R$ and $u_{0,R}=u_0$ and $\sigma_{R}=\sigma$ in $K_{R-1}$.
Finally, let $T\in (0,\infty]$ and $u_R\in C^{2;1}\big(B_{R^4}(0)\times(0,T)\big)\cap C^0\big(\overline{B_{R^4}(0)}\times [0,T)\big)$ be a solution of the initial boundary value problem
	\begin{equation} \label{eq mibvp}
	\begin{cases}
	\partial_t u_{R}(x,t)=g^{ij}(\sigma_{R},\nabla u_{R}) \,  \prescript{\sigma_R}{}\nabla^2_{ij} u_{R}(x,t) & \text{for} \;\; (x,t) \in B_{R^4}(0)\times (0,T) \,,\\
	u_{R}(x,0) = u_{0,{R}}(x) & \text{for} \;\; x\in B_{R^4}(0) \,,\\
	u_{R}(x,t) = 0 & \text{for} \;\; (x,t) \in \partial B_{R^4}(0)\times [0,T) \,.
	\end{cases}
	\end{equation}
	Then we obtain
	\[ 
		|\n u_R|_{{\sigma}_R}(x,t) \lesssim R^{-3n+\frac92} 
	\]
for all $|x| = R^4$ and $t \in [0,T)$.
\end{proposition}
\begin{proof}
	Let $b_{R}$ be a rotationally symmetric supersolution constructed in Section \ref{BarrierConstruction}.
	By Corollary \ref{cor barrier at infinity}, it can be chosen such that $b_{R}|_{\partial K_R}\geq \sup |u_0|+1$. Let $a_R$ be the value of $b_{R}$ at $\partial K_{R^4}$. Then $a_R\simeq R^{-3n+\frac{17}{2}}<1$ by Lemma \ref{barrierlemma} for $R\gg 1$. Let $\tilde{b}_R=b_{R}-a_R$. Then $\tilde{b}_R$ has the following properties:
	\begin{align*}
	\tilde{b}_R(x)\geq \sup_M |u_0| \;\; \text{for} \;\; |x|=R\,, \qquad \tilde{b}_R(x)=0 \;\; \text{for} \;\; |x|=R^4\,,
	\end{align*}
	where $x\in M$ is identified with a vector in $\R^n$ via the asymptotic chart. If $u_R$ is a solution of the Dirichlet problem of the proposition, we obviously have $|u_R(x,t)|\leq \sup_M |u_0|$ for all times. Therefore,
	\begin{align*}
    u_{R}(x,t)\leq \tilde{b}_{R}(x) \quad \text{for}\;\;|x|=R \;\; \text{and} \;\; |x|=R^4 \,, \qquad t\in [0,T) \,.
	\end{align*}
	By construction, $0=|u_{0,R}|\leq \tilde{b}_{R}$ on $K_{R^4}\setminus K_R$. By the maximum principle and because $b_{R}$ is a static supersolution, we obtain
		\[ 
	-\wt{b}_R \le u_R \le \wt{b}_R 
	\]
	on $\bigl(K_{R^4}\setminus K_R\bigr)\times [0,T)$ and we have equality on $\partial K_{R^4}\times[0,T)$.
	 This implies in particular
		\[ 
	\frac{\p (-\wt{b}_R)}{\p \nu}(x) \le \frac{\p u_R}{\p \nu}(x,t) \le \frac{\p \wt{b}_R}{\p \nu}(x) \quad \text{for} \quad x\in \partial K_{R^4}\,, \quad t\in (0,T) \,.
	\]
		Since by \eqref{eq asymptotics Db} we have
	\begin{align*}
	|\nabla b(x)| & \simeq \left(\frac{|x|}{R}\right)^{-n+\frac32} \,, 
	\end{align*}
	and since $|\nabla b|=|\nabla{\tilde{b}}|$, and because the derivatives tangential to the boundary vanish, we finally get
	\[ 
	|\n u_R|_{{\sigma}_R}(x,t) \lesssim R^{-3n+\frac92} \,. 
	\]
	for all $|x| = R^4$ and $t \in [0,T)$.
\end{proof}

\begin{remark}
	In the cases $n=1,2$, we also get bounds on the derivative at $|x|=R^4$. The fact that in these cases $b$ is unbounded as $|x|\to\infty$ does not play a role in the proof of the above lemma.
\end{remark}

\section{Derivative estimates}\label{destimates}

In this section, we are going to prove uniform derivative estimates for solutions of the modified initial boundary value problem \eqref{eq mibvp} for the mean curvature flow. These estimates are crucial as they directly yield long-time existence of the solutions. The key point is to compute an evolution inequality for a good quantity and to apply the maximum principle. Similar ideas were used by Gerhardt \cite{Ger00}.

Let $u$ evolve under mean curvature flow \eqref{eq:gMCF}.
Recall from \cite{EH91b}*{Proposition 3.2}, 
that the evolution of the quantity $v$ introduced in \eqref{definition_v} under the mean curvature flow is given by
\begin{equation}
	\label{eq:evoleqv}
	\left( \frac{d}{dt} - \Delta \right)v = -v \bigl( |\A|^2 + \Ric_{h}(\nu,\nu) \bigr) \,,
\end{equation}
where $\Ric_{h}$ is the Ricci tensor of the metric $h=\sigma_{ij} \d x^i \d x^j-(\d x^0)^2$, $A$ is the second fundamental form and $\nu$ is the unit normal of the hypersurface as defined in \eqref{definition_v}.
In this section, $\langle.,.\rangle$, $|.|$, $\nabla$ and $\Delta$ are taken with respect to the
induced metric $\gind$ on the graph of $u$ unless otherwise specified.

\begin{lemma}
	\label{lem:RicEst}
	Let $(M,\sigma)$ be a Riemannian manifold of bounded geometry and $u:M\to\R$ be a smooth function with spacelike graph in $(M\times \R,h)$. Then the Ricci tensor $\Ric_h$ of the metric $h$, evaluated on the unit normal $\nu$ of $\graph(u)$, satisfies
	\[
		\bigl| \Ric_{h}(\nu,\nu) \bigr| \le C (v^2-1)
	\]
	with some $C\ge 0$.
\end{lemma}

\begin{proof}
	Using the product structure of $M\times\Rb$, we can evaluate the Ricci tensor,
	\begin{align*}
		\Ric_{h}(\nu,\nu) &= \Ric_{\sigma} \left( \frac{\prescript{\s}{}\nabla u}{\sqrt{1-|\nabla u|^2_{\gM}}}, \frac{\prescript{\s}{}\nabla u}{\sqrt{1-|\nabla u|^2_{\s}}} \right) = v^2 \Ric_{\s}\bigl(\prescript{\s}{}\nabla u, \prescript{\s}{}\nabla u\bigr) \,,
	\end{align*}
	where $\prescript{\s}{}\n$ denotes the gradient with respect to $\s$.
	From the definition of $v$ in Eq.\ \eqref{definition_v} we infer $v^2|\nabla u|^2_{\s} = v^2 - 1$. 	
	Since $M$ has bounded geometry, its Ricci tensor satisfies $|\Ric_{\s}(w,w)|\le C|w|^2_{\gM}$ for some $C\ge 0$ and all $w$.
	Thus,
	\[
		\bigl| \Ric_{h}(\nu,\nu) \bigr| \le v^2 C |\nabla u|^2_{\gM} = C(v^2-1) \,. \qedhere
	\]
\end{proof}

In the sequel, we will study a function also considered by 
Gerhardt \cite{Ger00} (for similar ideas, see also e.\,g.\ \cite{EH91b}*{Proof of Proposition 4.4} or \cite{Lam17}*{Proof of Theorem 15}).
\begin{lemma}
	\label{lem:phievol}
Let $(M,\sigma)$ be a Riemannian manifold of bounded geometry and let $u\colon M\times [0,T)\to\R$ be a solution of the mean curvature flow in $(M\times\R,h)$. 
Then, the function $\varphi$ defined by
\begin{align}\label{def varphi}
	\varphi \coloneqq v \exp\bigl( \mu \ee^{\lambda u} \bigr), \,\qquad \lambda,\mu>0,
\end{align}
 evolves according to the equation
	\begin{align*}
		\left(\frac{d}{dt} - \Delta\right)\varphi &= - \varphi \Bigl( |\A|^2 + \Ric_{h}(\nu,\nu) \Bigr) - 2\lambda \mu \frac{\varphi}{v} \ee^{ \lambda u}\< \nabla u, \nabla v\> \\
			& \qquad - \mu \lambda^2 \varphi \ee^{\lambda u} \Bigl\{ \mu \ee^{\lambda u} + 1 \Bigr\} |\nabla u|^2 \,.
	\end{align*}
\end{lemma}

\begin{proof}
	For the first derivative of $\varphi$, 
	we calculate
	\begin{equation}
		\label{eq:phiDeriv}
		\varphi_j = v_j \exp\bigl( \mu \ee^{\lambda u} \bigr) + \mu \lambda \varphi \ee^{\lambda u} u_j \,.
	\end{equation}
	Further, for the second derivative we get
	\begin{align*}
		\nabla^2_{ij}\varphi &= \bigl( \nabla_{ij}^2 v\bigr) \exp\bigl( \mu \ee^{\lambda u} \bigr) + \mu \lambda \exp\bigl( \mu \ee^{\lambda u} \bigr) \ee^{\lambda u} u_i v_j \\
			& \quad + \mu \lambda \bigl( v_i \exp(\mu \ee^{\lambda u} ) + \mu \lambda \varphi \ee^{\lambda u} u_i \bigr) \ee^{\lambda u} u_j \\
			& \quad + \mu \lambda^2 \varphi \ee^{\lambda u} u_i u_j + \mu \lambda \varphi \ee^{\lambda u} \nabla^2_{ij} u \,.
	\end{align*}
	Taking the trace with respect to $\gind$, we obtain
	\begin{align*}
		\Delta \varphi &= \frac{\varphi}{v} \Delta v + 2 \lambda \mu \exp\bigl( \mu \ee^{\lambda u} \bigr) \ee^{\lambda u} \<\nabla u, \nabla v\> \\
			& \qquad + \mu^2\lambda^2 \varphi \ee^{2\lambda u} |\nabla u|^2 + \mu \lambda^2 \varphi \ee^{\lambda u} |\nabla u|^2 + \mu \lambda \varphi \ee^{\lambda u} \Delta u \,.
	\end{align*}
	Since the time derivative of $\varphi$ is given by
	$\partial_t \varphi = (\partial_t v) \frac{\varphi}{v} + \mu\lambda \varphi \ee^{\lambda u} \partial_t u$,
    	Lemma \ref{lem:evu} implies that $\varphi$ evolves according to
	\begin{align*}
		\left(\frac{d}{dt} - \Delta\right)\varphi &= \frac{\varphi}{v}\left(\frac{d}{dt} - \Delta \right)v + \mu\lambda \varphi \ee^{\lambda u}\underbrace{\left( \frac{d}{dt} - \Delta\right) u}_{=0} - 2\lambda \mu \frac{\varphi}{v} \ee^{\lambda u}\< \nabla u, \nabla v\> \\
			& \qquad - \mu \lambda^2 \varphi \ee^{\lambda u} \Bigl\{ \mu \ee^{\lambda u} + 1 \Bigr\} |\nabla u|^2 \,.
	\end{align*}	
	The claim now follows from the evolution equation \eqref{eq:evoleqv} for $v$.
\end{proof}

\begin{lemma} \label{mixed_term_est}
Let $(M,\sigma)$ be a Riemannian manifold of bounded geometry and let $u\colon M\to\R$
be strictly spacelike. Then, the estimate
	\[
		\bigl| \langle \nabla u, \nabla v\rangle \bigr| \le |A| |\nabla u|^2
	\]
	holds.
\end{lemma}

\begin{proof}
	Using the relation $v^2=1+|\nabla u|^2$ from Lemma \ref{uniform_metrics}, we calculate
	\[
		\partial_i (v^2) = 2v v_i = 2 g^{jk} (\nabla^2_{ij} u) u_k \,.
	\]
	Consequently,
	\[
		\bigl| \<\nabla u, \nabla v\> \bigr| = \frac{1}{v} \bigl| g^{jk} g^{il} (\nabla_{ij}^2 u) u_k u_l \bigr| \le \frac{1}{v} |\nabla^2 u| |\nabla u|^2 \,.
	\]
	Let $\pi_{\Rb}\colon M\times\Rb\to\Rb$ denote the projection onto the second factor.
	Then the claim follows from 
	\[
		|\nabla^2 u| = |\d\piR(\nabla^2 F)| = |\d\piR(A)| = |A| v \,. \qedhere
	\]
\end{proof}

\begin{lemma}\label{lem_derivative_estimate}
	Let $u_R$ be a solution of the modified initial boundary value problem \eqref{eq mibvp}.
	Then $\sup_{M\times\{t\}}v\le c$ (where $v$ is computed as in \eqref{definition_v} using $u_R$ instead of $u$)
	 for some finite constant $c\ge 1$ which depends on $\sigma$ and the initial data, but 
	is independent of $R\gg 1$.
\end{lemma}

\begin{proof}
	Note that as long as the solution exists, we have 
    	\[
    		\min u_0\leq \min u_{0,R}\leq u_{t,R}\leq\max u_{0,R} \leq\max u_0
	\]
	by the avoidance principle and the construction of $u_{0,R}$. After possibly applying a 
	shift to $u_0$, we further may assume
	\[
		\min u_0 \ge 0 \,.
	\]
	
	Using Lemma \ref{mixed_term_est}
	and $|\nabla u_R|^2\le v^2$ (see Lemma \ref{uniform_metrics}), we obtain the estimate
	\[
		2 \<\nabla u_R, \nabla v\> \le \varepsilon |A|^2 v + \frac{|\nabla u_R|^4}{\varepsilon v} \le v \left( \varepsilon |A|^2 + \frac{|\nabla u_R|^2}{\varepsilon} \right)
	\]
	for any $\varepsilon>0$.
	Together with the Lemmas \ref{lem:RicEst} and \ref{lem:phievol}, this implies
	\begin{align*}
		\left(\frac{d}{dt} - \Delta\right) \varphi &\le -\varphi \left( \left( 1 - \varepsilon \lambda \mu \ee^{\lambda u_R} \right) |A|^2 - C(v^2-1) \right) \\
			& \qquad - \lambda \mu \varphi \ee^{\lambda u_R} \left( \lambda\bigl(\mu \ee^{\lambda u_R}+1\bigr) - \frac{1}{\varepsilon} \right) |\nabla u_R|^2 \,.
	\end{align*}
	Now let $\varepsilon\coloneqq\ee^{-\lambda u_R}$, $\lambda\coloneqq C$ and $\mu\coloneqq \frac{1}{C}$. Then
	$\varepsilon\lambda\mu\ee^{\lambda u_R} = 1$, and using $|\nabla u_R|^2=v^2-1$ we calculate
	\begin{align*}
		\left(\frac{d}{dt} - \Delta\right)\varphi &\le \varphi C(v^2-1) - \varphi \ee^{Cu_R} \bigl( \ee^{Cu_R} + C - \ee^{C u_R} \bigr)(v^2-1) \\
			&= -\varphi C\bigl( \ee^{Cu_R} - 1 \bigr) (v^2-1) \stackrel{\min u_R \ge 0}\le 0 \,.
	\end{align*}
	By the gradient estimates in Proposition \ref{boundary_gradient_thm}, the function $\varphi$ is bounded at $\partial B_{R^4}(0)$ for positive times. We can thus apply the maximum principle to conclude that
	 \[
		 \sup_{B_{R^4}(0)\times[0,T)}\varphi\le \max\left\{\sup_{B_{R^4}(0)\times\left\{0\right\}}\varphi,\sup_{\partial B_{R^4}(0)\times[0,T)}\varphi\right\} \,,
	 \]
	 which yields the statement of the lemma.
\end{proof}

\begin{lemma} \label{lem no lift-off for constructed_2nd_version}
	Let $R>0$ be large enough, $T\in (0,\infty]$ and let $(M^n,\sigma)$, $n\geq3$, be an asymptotically flat manifold.
	Suppose $u_R\in C^{2;1}\big(K_R\times(0,T)\big)\cap C^0\big(\overline{K_R}\times [0,T)\big)$ is a solution of the modified initial boundary value problem
	\begin{equation*}
		\begin{cases}
			\partial_t u_R(x,t)=g^{ij}(\sigma_R,\nabla u_R) \, \prescript{\s}{}\nabla^2_{ij} u_R(x,t) & \text{for} \;\; (x,t) \in K_R\times (0,T) \,,\\
			u_R(x,0) = u_{0,R}(x) & \text{for} \;\; x\in K_R \,,\\
			u_R(x,t) = 0 & \text{for} \;\; (x,t) \in \partial K_R\times [0,T] \,,
		\end{cases}
	\end{equation*}
	where $u_{0,R}$ and $\sigma_R$ are the modified initial data provided by Proposition \ref{interpolation_thm}.
	
	We extend $u_{0,R}$ 
	and $u_R$ by zero for $|x|>R$.
	We additionally assume that we are given $\eps>0$ and $r_1>0$ such that $|u_0(x)|\le\eps$ for $|x| \ge r_1$.
	
	Then there is a function $b_\eps \colon M\to \R_+$, dependent only on  $\sigma$, $\eps$, $r_1$, and $\sup_{M} u_0$ and a small constant $\kappa>0$  dependent only on  $\sigma$, $r_1$, $\sup_{M}\bigl( |u_0|+|\nabla u_0|_{\sigma}\bigr)$, such that for all large $R$
	\begin{enumerate}[leftmargin=1cm,label=(\roman*)]
		\item $|u_R(x,t)| \le b_\eps(x)$ for $(x,t)\in M\times (0,T)$,
		\item $b_\eps(x)\to \eps \text{ for } |x| \to \infty$,
		\item $|\nabla u_R(x,t)|_{\sigma} \le 1-\kappa$ for $(x,t)\in M\times (0,T)$.
	\end{enumerate}
Consequently, each solution $u_R$ can be extended for all times and satisfies the estimates (i) and (iii) uniformly for all $t\in (0,\infty)$.
\end{lemma}

\begin{proof}Choose $b_{\varepsilon}$ as in the proof of Lemma \ref{lem no lift-off for constructed}.
The estimates (i) and (ii) hold due to Remark \ref{rem no lift-off_AEcase}. Part (iii) follows directly from Lemma \ref{lem_derivative_estimate}. Note that all these estimates are independent of $T$. Suppose that $T_{\Max}<\infty$ is the maximal time of existence. By standard theory of parabolic equations, we also have  bounds on all derivatives of $u_R$ on $(T_{\Max}/2,T_{\Max})$ which are uniform up to $T_{\Max}$. Therefore $u_{t,R}$ converges in all derivatives to some function  $u_{T_{\Max},R}$ and we can apply short-time existence to get a smooth extension of $u_R$ beyond $T_{\Max}$ which contradicts the maximality of $T_{\Max}$.
\end{proof}

\begin{remark} \label{rem no lift-off_2nd_version}
	Suppose $r_1$ is given as a function of $\eps$ in the sense that for all $\eps>0$ $|u_0(x)|<\eps$ for $|x|\ge r_1(\eps)$.
	Here we may take the infimum $b(x)\coloneqq \inf_{\eps>0}b_{\eps}(x)$ which depends on $\sigma$, $\sup_{M} u_0$, and the function $r_1$. With $\kappa>0$ as above and $R\gg 1$ large enough, we then have
	\begin{enumerate}[leftmargin=1cm,label=(\roman*)]
		\item $|u_R(x,t)| \le b(x)$ for $(x,t)\in M\times (0,T)$, 
		\item $b(x)\to 0 \text{ for } |x| \to \infty$
			\item $|\nabla u_R(x,t)|_{\sigma_R} \le 1-\kappa$ for $(x,t)\in M\times (0,T)$.
	\end{enumerate}
\end{remark}

\section{Long-time existence and convergence}\label{mainresults}

In this section, we are going to prove the main statements of the paper.

\begin{theorem}\label{thm_longtime_existence}
	Let $(M^n,\sigma)$ be an asymptotically flat manifold with $n\geq3$ and $u_0\in C^{0,1}(M)$ be uniformly spacelike such that
	\[
		u_0(x)\to 0 \quad \text{for} \quad |x|\to\infty.
	\]
	Then there exists a unique uniformly spacelike solution of \eqref{eq MCF} with $u(\cdot,0)=u_0$ which exists for all times and satisfies $u\in C^{2;1}\bigl(M\times (0,\infty) \bigr)\cap C^0\bigl(M\times [0,\infty)\bigr)$.
	Moreover, $u$ satisfies
	\begin{equation}\label{decay_thm_longtime_existence}
		u(x,t)\to 0 \quad \text{for} \quad |x|\to \infty.
	\end{equation}
	and this convergence is uniform in $t$. More precisely, there is a function $b\colon M\to \mathbb{R}$ and a constant $\kappa>0$, both only dependent on $u_0$ and $\sigma$, such that
	\begin{enumerate}[leftmargin=1cm,label=(\roman*)]
		\item $|u(x,t)|\le b(x)$ for $(x,t)\in M\times (0,\infty)$,
		\item $b(x)\to 0$ for $|x|\to \infty$,
		\item $|\nabla u(x,t)|_{\sigma} \le 1-\kappa$ for $(x,t)\in M\times (0,\infty)$.
	\end{enumerate}
\end{theorem}

\begin{proof}By Lemma \ref{lem no lift-off for constructed_2nd_version}, the solutions $u_R$ exist for all times
	and are uniformly spacelike in $T$ and $R$.
	By standard theory for parabolic equations, we get uniform bounds on higher derivatives on
	 $(\delta,\infty)$ where the bounds only depend on $u_0$, $\sigma$ and $\delta$ and are in particular independent of $R$. Therefore, by the Arzelà-Ascoli theorem, the family $u_R$ subconverges in all derivatives to a solution $u$ of \eqref{eq MCF} with initial data $u_0$. The estimates (i)--(iii) are consequences of the estimates (i)--(iii) in Lemma \ref{lem no lift-off for constructed_2nd_version} and Remark \ref{rem no lift-off_2nd_version}.
	
	To show uniqueness, one uses \eqref{decay_thm_longtime_existence} in the following sense: Suppose that $v(x,t)$ is another uniformly spacelike solution of \eqref{eq MCF} which is defined up to a time $T$. Then by Theorem \ref{thm no lift-off}, $v$ also satisfies \eqref{decay_thm_longtime_existence} for each $t\in[0,T)$. Therefore, $u\pm \varepsilon$  can be used as a barrier for $v$ on $M\times [0,T)$. The maximum principle implies that for each $\varepsilon>0$, we get $u-\varepsilon\leq v\leq u+\varepsilon$. Uniqueness follows from $\varepsilon\to0$.
\end{proof}

\begin{remark}\label{rem_longtime_existence}
	Statement \eqref{decay_thm_longtime_existence} also follows in dimensions $n=1,2$ but we cannot conclude (i) and (ii) in these cases. These properties are however essential in the proof of the main theorem below. In the case $n=1$, we give a separate proof that is done in a subsection below.
\end{remark}

\begin{theorem}\label{thm_longtime_convergence}
	Let $u$ be the solution in Theorem \ref{thm_longtime_existence}. Then $u\left(\cdot,t\right)\xrightarrow[t\to\infty]{}0$
	 uniformly in $C^{l}$ for all $l\in\mathbb{N}_{0}$. More precisely,
	for each $l\in\N_0$, there exists a function $v_l\colon [0,T)\to \R$ with $v_l(t)\to 0$ as $t\to\infty$ such that
	\[
		\bigl\| u_t\bigr\|_{C^{l}(M)}\leq v_l(t) \,.
	\]
\end{theorem}

\begin{proof}
	First we prove the statement for $l=0$. Since $u$ decays as $x\to\infty$ by
	Theorem \ref{thm_longtime_existence}, we can apply the maximum principle to conclude that
	\[
		v\colon [0,T)\to\Rb\,, \qquad v(t)\coloneqq \sup_{x\in M} |u(x,t)| \,,
	\]
	is monotonically decreasing. If $v(t)\to 0$ for $t\to\infty$, the claim follows. Assume that $v\left(t\right)\xrightarrow[t\to\infty]{}\delta>0$.
	By Theorem \ref{thm_longtime_existence} there exists $R>0$ such that $\abs{u\left(x,t\right)}\leq\frac{\delta}{2}$ for all $\abs x\geq R$
	and $t\geq0$. Using the strict maximum principle once again in $\overline{B_{R}\left(0\right)}\times[0,\infty)$
	yields that either $v$ is strictly decreasing or $u\left(\cdot,t\right)$
	is constant for all $t\geq T_{0}$ for some $T_0\ge 0$. In the second case we conclude
	(again by Theorem \ref{thm_longtime_existence}) that $u\left(\cdot,t\right)=0$ for all $t\geq T_{0}$.

	Let $t_{k}\geq0$ with $t_{k}\xrightarrow[k\to\infty]{}\infty$ be
	arbitrary. Due to the uniform estimates for $u$ in $C^{l}$ for $l\in\mathbb{N}$
	it follows by the Arzel\`a-Ascoli theorem (after choosing a subsequence also labelled
	$\left(t_{k}\right)_{k\in\mathbb{N}}$) that
	\[
		u\left(\cdot,t+t_{k}\right)\xrightarrow[k\to\infty]{}\tilde{u}\left(\cdot,t\right)
	\]
	uniformly and $\tilde{u}$ is a solution of the same differential
	equation as $u$. Due to the uniform convergence we conclude that
	\begin{align}\label{max}
		\sup_{x\in M}\abs{\tilde{u}\left(x,t\right)} & =\lim_{k\to\infty}\sup_{x\in M}\abs{u\left(x,t+t_{k}\right)} =\lim_{k\to\infty}v\left(t+t_{k}\right) =\delta>0
	\end{align}
	for all $t\geq0$. Because $\abs{u\left(x,t\right)}\xrightarrow[\abs x\to\infty]{}0$ uniformly in $t$ the same holds for $\tilde{u}$.
	Therefore, $x\mapsto \tilde{u}(x,t)$ attains a maximum for each $t\geq0$. By the strict maximum principle, this is strictly decreasing unless $\tilde{u}$ is constant, hence identically zero. However, this is a contradiction to \eqref{max}.

	For arbitrary $l\in\mathbb{N}$, the statement follows immediately
	from the following interpolation inequality
	\[
		\norm{\nabla f}_{C^{0}\left(M \right)}^{2}\leq c\left(n\right)\norm f_{C^{0}\left(M\right)}\bigl\|f\bigr\|_{C^{2}\left(M\right)}
	\]
	since $\norm{u\left(\cdot,t\right)}_{C^{0}\left(M\right)}\xrightarrow[t\to\infty]{}0$
	and the uniform bounds for the derivatives of $u$ hold.
\end{proof}

\subsection{The One-Dimensional Case}
For $(M,\sigma)=(\r,\delta)$, 
the mean curvature flow can be written in a particularly simple form which allows us to prove our main result with a different method. To be more precise, for maps $u\colon \Rb\times[0,T)\to\Rb$, the graphical mean curvature flow \eqref{eq:gMCF}
may equivalently be written as
\begin{equation}
	\label{eq:1dMCF}
	\partial_t u = \frac{u''}{1-\bigl(u'\bigr)^2} = \frac{1}{2} \left[ \ln \frac{1+u'}{1-u'} \right]'\,.
\end{equation}
From Theorem \ref{thm_longtime_existence} and Remark \ref{rem_longtime_existence}, we know already that $\eqref{eq:1dMCF}$ admits uniformly spacelike solutions that exist for all times, provided that the initial data is nice enough. In the following, we will show convergence to $0$ for uniformly spacelike $u_0$ and we will even get a convergence rate if $u_0\in L^2(\R)$.

\begin{lemma}
	\label{lem:1dBounded}
	Let $u\colon \R\times (0,T)\to\infty$ be a smooth solution of \eqref{eq:1dMCF} (which is continuous up to $t=0$) with uniformly spacelike initial data $u_0\in C^{0,1}(\R)\cap L^2(\R)$. Then $u_t\in L^2(\R)$ and $\left\|u_t\right\|_{L^2}\leq \left\|u_0\right\|_{L^{2}}$.
\end{lemma}

\begin{proof}
	Choose a smooth cutoff function $\psi$ with $0\leq \psi\leq 1$,
	\[
		\psi(x) =
		\begin{cases}
			1 \,, & x\in[-1,1] \,, \\
			0 \,, & x\notin (-2,2) \,,
		\end{cases} \qquad \text{and} \qquad |\psi'|\leq 2 \,.
	\]
	Let $\varphi(x)=\psi(x/R)$ where $R>0$. Then we get
	\begin{align*}
		\partial_t\int_{\Rb}\varphi^2u^2\der x&= 2\int_{\Rb}\varphi^2(\partial_t u) u\der x=\int_{\Rb}\varphi^2u\left(\ln\frac{1+u'}{1-u'}\right)'\der x\\
			&=-2\int_{\Rb}\varphi'\varphi u\left(\ln\frac{1+u'}{1-u'}\right)\der x -\int_{\Rb}\varphi^2u'\left(\ln\frac{1+u'}{1-u'}\right)\der x\\
			&\le\frac{1}{\delta}\int_{\Rb}(\varphi')^2u^2\der x+\delta\int_{\Rb}\varphi^2\left(\ln\frac{1+u'}{1-u'}\right)^2\der x \\
				& \qquad -\int_{\Rb}\varphi^2u'\left(\ln\frac{1+u'}{1-u'}\right)\der x ,
	\end{align*}
	where $\delta>0$ is fixed below.
	Also using $x\ln \frac{1+x}{1-x} \ge x^2$ for $|x|<1$, we further estimate
	\begin{align*}
			\partial_t\int_{\Rb}\varphi^2u^2\der x&\leq\frac{4}{\delta R^2}\int_{-2R}^{2R}u^2\der x+ C \delta\int_{\Rb}\varphi^2(u')^2\der x -\int_{\Rb}\varphi^2(u')^2\der x\\ &\leq \frac{4}{\delta R^2}\int_{-2R}^{2R}u^2\der x \,,
	\end{align*}
	where $C$ depends on $\sup |u'_t|\leq1-\varepsilon$ and $\delta\in (0,C^{-1})$.
	Now, define
	\begin{align*}
		A(t,R)=\sup_{y\in\Rb}\int_{\Rb}u^2\varphi_{R,y}^2\der x \,,
	\end{align*}
	where $\varphi_{R,y}(x)=\psi((x-y)/R)$. Note that $A(t,R)$ is bounded because $u$ is. After integration in time, we get
	\begin{align*}
		\int_{\Rb}\varphi_{R,y}^2u_t^2\der x&\leq \int_{\Rb}\varphi_{R,y}^2u_0^2\der x+\frac{4}{\delta R^2}\int_0^t \int_{-2R}^{2R}u_s^2\der x \der s \\
		&\leq A(0,R) +\frac{8}{\delta R^2}\int_0^tA(s,R)\der s
	\end{align*}
	and taking the supremum over $x$ on the left-hand side yields
	\begin{align*}
		A(t,R)\leq A(0,R)+\frac{8}{\delta R^2}\int_0^tA(s,R)\der s
	\end{align*}
	and by the Gronwall inequality, we obtain 
	\begin{align*}
		A(t,R)\leq A(0,R)\exp\left(\frac{8t}{\delta R^2}\right)
	\end{align*}
	and the result follows by letting $R\to\infty$.
\end{proof}

Now we extend the cutoff argument to get a uniform estimate of the $H^1$-norm.
\begin{lemma}
	\label{lem:1dDerivDecay}
	Let $u\colon \R\times (0,T)\to\infty$ be a smooth solution of \eqref{eq:1dMCF} with uniformly spacelike initial data $u_0\in C^{0,1}(\R)\cap L^2(\R)$. Then  $u_t\in H^1(\Rb)$ and
	\[
		\left\|u_t\right\|_{L^2}^2+t \left\|u_t'\right\|_{L^2}^2\leq \left\|u_0'\right\|_{L^2}^2
	\]
	for all $t>0$.
\end{lemma}

\begin{proof}
	Let $\varphi$ be  as in the proof of Lemma \ref{lem:1dBounded}. Then 
	\begin{align*}
		\partial_t	&\left(t \int_{\Rb}\varphi^2(u')^2\der x\right)=\int_{\Rb}\varphi^2(u')^2\der x+2t\int_{\Rb}\varphi^2(u')(\partial_t u)'\der x\\
			&=\int_{\Rb}\varphi^2(u')^2\der x+2t\int_{\Rb}\varphi^2u'\left(\frac{u''}{1-(u')^2}\right)'\der x\\
			&=\int_{\Rb}\varphi^2(u')^2\der x -2t\int_{\Rb}\varphi^2\frac{(u'')^2}{1-(u')^2}\der x-4t\int_{\Rb}\varphi'\varphi u'\left(\frac{u''}{1-(u')^2}\right)\der x\\
			&\le \int_{\Rb}\varphi^2(u')^2\der x -2t\int_{\Rb}\varphi^2\frac{(u'')^2}{1-(u')^2}\der x\\
				&\qquad +\frac{2t}{\delta}\int_{\Rb}(\varphi')^2(u')^2\der x+2t\delta \int_{\Rb}\varphi^2\frac{(u'')^2}{(1-(u')^2)^2}\der x\\
			&\leq \int_{\Rb}\varphi^2(u')^2\der x+\frac{2t}{\delta}\int_{\Rb}(\varphi')^2(u')^2\der x -2t(1-C\delta)\int_{\Rb}\varphi^2\frac{(u'')^2}{1-(u')^2}\der x\\
			&\leq \int_{\Rb}\varphi^2(u')^2\der x+\frac{8t}{R^2\delta}\int_{-2R}^{2R}(u')^2\der x -2t(1-C\delta)\int_{\Rb}\varphi^2\frac{(u'')^2}{1-(u')^2}\der x \,,
	\end{align*}
	where $C$ depends on $\sup |u'_t|<1$. Now, pick $\alpha\in(0,1)$. Then, combining the above estimate
	with the one from the Lemma \ref{lem:1dBounded}, we obtain for any $\delta\in (0,(1-\alpha)C^{-1})$
	\begin{align*}
		\partial_t&\left(\int_{\Rb}\varphi^2u^2\der x+\alpha t \int_{\Rb}\varphi^2(u')^2\der x\right)
	\leq\frac{4}{\delta R^2}\int_{-2R}^{2R}u^2\der x+\frac{8\alpha t}{R^2\delta}\int_{-2R}^{2R}(u')^2\der x\\
			&\qquad -(1-\alpha-C\delta)\int_{\Rb}\varphi^2(u')^2\der x-2t(1-C\delta) \alpha \int_{\Rb}\varphi^2\frac{(u'')^2}{1-(u')^2}\der x\\
		&\leq \frac{4}{\delta R^2}\int_{-2R}^{2R}u^2\der x+\frac{8\alpha t}{R^2\delta}\int_{-2R}^{2R}(u')^2\der x \,.
	\end{align*}
	Analogously to the above, we define
	\begin{align*}
		A(t,R,\alpha)=\sup_{y\in\Rb}\left(\int_{\Rb}u^2\varphi_{R,y}^2\der x+\alpha t\int_{\Rb}(u')^2\varphi_{R,y}^2\der x\right) \,,
	\end{align*}
	where $\varphi_{R,y}(x)$ as in the proof of the previous lemma. After integration in time, we get
	\begin{align*}
		& \int_{\Rb}\varphi_{R,y}^2u_t^2\der x+\alpha t\int_{\Rb}\varphi_{R,y}^2(u'_t)^2\der x \\
			&\qquad \leq \int_{\Rb}\varphi_{R,y}^2u_0^2\der x
	+\frac{4}{\delta R^2}\int_0^t \int_{-2R}^{2R}u_s^2\der x\der s + \frac{8\alpha }{R^2\delta}\int_0^ts\int_{-2R}^{2R}(u'_s)^2\der x\der s  \\
			& \qquad \leq A(0,R,\alpha)+\frac{16}{\delta R^2}\int_0^tA(s,R,\alpha)\der s
	\end{align*}
	and taking the supremum over $x$ on the left-hand side yields
	\begin{align*}
		A(t,R,\alpha)\leq A(0,R,\alpha)+\frac{16}{\delta R^2}\int_0^tA(s,R,\alpha)\der s.
	\end{align*}
	By the Gronwall inequality, we obtain 
	\begin{align*}
		A(t,R,\alpha)\leq A(0,R,\alpha)\exp\left(\frac{16t}{\delta R^2}\right)
	\end{align*}
	and the result follows from letting $R\to\infty$ and then letting $\alpha\to 1$.
\end{proof}

\begin{proposition}\label{lem:1dConv}
	Let $u\colon \R\times (0,T)\to\infty$ be a smooth solution of \eqref{eq:1dMCF} with uniformly spacelike initial data $u_0\in C^{0,1}(\R)\cap L^2(\R)$.
	 Then $u_t\to0$ for $t\to\infty$,  uniformly in all derivatives. More precisely,
	\[
		\left\|u^{(k)}\right\|_{L^{\infty}}\leq C_k\cdot t^{-\frac{1}{4}}
	\]
	for some $C\ge 0$ and for all $t>0$.	
\end{proposition}
	
\begin{proof}
	By the one-dimensional  Gagliardo-Nirenberg inequality and Lemma \ref{lem:1dDerivDecay}, we obtain
	\[
		\left\|u\right\|_{L^{\infty}}\leq C\left\|u'\right\|_{L^2}^{\frac{1}{2}}\left\|u\right\|_{L^2}^{\frac{1}{2}}
		= Ct^{-\frac{1}{4}}(t^{\frac{1}{2}}\left\|u'\right\|_{L^2})^{\frac{1}{2}}\left\|u\right\|_{L^2}^{\frac{1}{2}} \leq Ct^{-\frac{1}{4}}\,.	
	\]
	For higher derivatives, the assertion follows from short-time derivative estimates of the form 
	\[
		\left\| u^{(k)}_{t+1}\right\|_{L^{\infty}}\leq C_k\cdot \left\| u_t\right\|_{L^{\infty}} \,,
	\]
	which finishes the proof.
\end{proof}

\begin{remark}
	The convergence rate in Proposition \ref{lem:1dConv} is the same as one gets for solutions of the $1$-dimensional heat equation with initial data in $L^2$. This can be seen as follows: If $u_0\in L^2(\R)$, the solution of the heat equation with initial data $u_0$ is given by 
	\begin{align*}
	u(x,t)=\int_{\R}K(y-x,t)u_0(y)\d y\,,\qquad K(x,t)=(4\pi t)^{-1/2}\cdot e^{-\frac{x^2}{4t}}\,,
	\end{align*}
	where $K$ is the $1$-dimensional heat kernel. A direct computation proves that the inequality $\left\|K(\cdot,t)\right\|_{L^2}\leq C\cdot t^{-1/4}$ holds. Finally Young's inequality for convolutions implies that
	\begin{align*}
	\left\|u(\cdot,t)\right\|_{L^{\infty}}\leq \left\|K(\cdot,t)\right\|_{L^2}\left\|u_0\right\|_{L^2}\leq C\cdot t^{-1/4}\cdot \left\|u_0\right\|_{L^2}\,,
	\end{align*}
	which is the estimate we claimed.
\end{remark}

\begin{theorem}
	Let $u\colon\R\times (0,T)\to\infty$ be a smooth, uniformly spacelike solution of \eqref{eq:1dMCF} with uniformly spacelike initial data 
	$u_0\in C^{0,1}(\R)$ such that $u_0(x)\to0$ as $|x|\to\infty$. Then, if $t\to\infty$, 
	$u_t\to0$ uniformly in 
	all derivatives. More precisely, there exist functions $v_k(t)$ with $v_k(t)\to 0$ as $t\to\infty$ such that
	\[
		\left\|u^{(k)}\right\|_{L^{\infty}}\leq v_k(t)
	\]
	for some $C\ge 0$ and for all $t>0$.		
\end{theorem}
	
\begin{proof}
	For each $\varepsilon>0$, pick a compactly supported smooth function $v_{0,\varepsilon}$ such that 
	\begin{align*}
		-v_{0,\varepsilon}-\frac{\varepsilon}{2}\leq u_0\leq v_{0,\varepsilon}+\frac{\varepsilon}{2}\,.
	\end{align*}
	By the maximum principle, the corresponding solutions satisfy
	\begin{align*}
		-v_{\varepsilon}-\frac{\varepsilon}{2}\leq u\leq v_{\varepsilon}+\frac{\varepsilon}{2}
	\end{align*}
	for all times. As $v_{\varepsilon}\to 0$, we conclude that $|u|\leq \varepsilon$ for all $t\geq T(\varepsilon)$ 
	and some large time $T(\varepsilon)$, and the assertion for $u$ follows for $k=0$. For the higher derivatives, 
	the claim follows from standard estimates.
\end{proof}

\end{document}